\documentclass[11pt]{amsart}
\usepackage{setspace}
\usepackage{amsmath,amsfonts,amssymb,amsthm}
\usepackage{anysize}
\usepackage{tikz-cd}
\marginsize{2cm}{2cm}{2cm}{2cm}
\usepackage{graphicx}
\usepackage{color}
\usepackage[pdftex]{hyperref}
\usepackage{soul}
\usepackage{enumerate}
\usepackage{caption}

\newenvironment{smallarray}[1]
 {\null\,\vcenter\bgroup\scriptsize
  \arraycolsep=.13885em
  \hbox\bgroup$\array{@{}#1@{}}}
 {\endarray$\egroup\egroup\,\null}

\def\z{\mathfrak{z}}
\def\u{\mathfrak{u}}

\def\l{\mathfrak l}
\def\g{\mathfrak{g}}
\def\h{\mathfrak{h}}

\def\a{\mathfrak{a}}
\def\q{\mathfrak{q}}
\def\s{\mathfrak s}

\def\cJ{\mathcal{J}}
\def\cC{\mathcal{C}}
\def\cA{\mathcal{A}}
\def\cB{\mathcal{B}}
\def\cN{\mathcal{N}}

\def\hcx{\{J_{\alpha}\}}

\def\C{\mathbb{C}}

\def\R{\mathbb{R}}
\def\Q{\mathbb{Q}}
\def\Z{\mathbb{Z}}
\def\N{\mathbb{N}}
\def\H{\mathbb H}
\def\J{\mathbb J}

\def\al{\alpha}
\def\be{\beta}
\def\ga{\gamma}
\def\e{\operatorname{e}}

\def\ad{\operatorname{ad}}
\def\tr{\operatorname{tr}}

\def\alt{\raise1pt\hbox{$\bigwedge$}}
\def\pint{\langle \cdotp,\cdotp \rangle }

\theoremstyle{plain}
\newtheorem{theorem}{\bf Theorem}[section]
\newtheorem{corollary}[theorem]{\bf Corollary}
\newtheorem{proposition}[theorem]{\bf Proposition}
\newtheorem{lemma}[theorem]{\bf Lemma}

\theoremstyle{definition}

\newtheorem{example}[theorem]{\bf Example}

\theoremstyle{remark}
\newtheorem{remark}[theorem]{\bf Remark}

\newcommand{\ri}{{\rm (i)}}
\newcommand{\rii}{{\rm (ii)}}
\newcommand{\riii}{{\rm (iii)}}

\title[Applications of the quaternionic Jordan form]{Applications of the quaternionic Jordan form to hypercomplex  geometry}

\author{Adrián Andrada}
\email{adrian.andrada@unc.edu.ar}
\author{María Laura Barberis}
\email{mlbarberis@unc.edu.ar}

\date{}
\address{FAMAF, Universidad Nacional de C\'ordoba and CIEM-CONICET, Av. Medina Allende s/n, Ciudad Universitaria, X5000HUA C\'ordoba, Argentina}

\thanks{This work was partially supported by CONICET, SECyT-UNC and ANPCyT (Argentina)}


\begin{document}

\begin{abstract} 
We apply the quaternionic Jordan form to classify the nilpotent hypercomplex almost abelian Lie algebras in all dimensions and to carry out the complete classification of  12-dimensional hypercomplex  almost abelian Lie algebras. Moreover, we determine which 12-dimensional simply connected hypercomplex  almost abelian Lie groups admit lattices. Finally, for each integer $n>1$ we construct  infinitely many, up to diffeomorphism,  $(4n+4)$-dimensional hypercomplex almost abelian solvmanifolds which are completely solvable. These solvmanifolds arise from a distinguished family of monic integer polynomials of degree~$n$. 
\end{abstract}

\maketitle

\section{Introduction}
The problem of existence of left invariant  geometric structures on Lie groups is an active field of research. In the particular case when the Lie group is almost abelian,  several authors 
have made important  recent contributions to the subject  \cite{AT, Av, BeFi, C-M, CM, FiPa1, FiPa2, HO, MeTa, Mo}. Almost abelian Lie groups also have interesting applications in theoretical physics (see for instance \cite{AV, RS}). 
Recall that a Lie group is called almost abelian when its  Lie algebra has a codimension one abelian ideal, that is, it  can be written as $\g=\R e_0\ltimes_A \R^{d}$, where the matrix $A\in \mathfrak{gl}(d,\R)$ encodes the adjoint action of $e_0$ on the abelian ideal $\R^{d}$. The existence of a left invariant geometric  structure on $G$ imposes restrictions on $A$.
 
In this paper we focus on a special type of geometric structures, namely, hypercomplex structures. A hypercomplex structure on a smooth manifold $M$ is a triple $\{J_1, J_2, J_3\}$ of complex structures satisfying the laws of the quaternions. Hypercomplex structures  are present in many
branches of theoretical and mathematical physics \cite{HKLR, V}. It was proved by Boyer \cite{Boy} that, in real  dimension 4, the only compact hypercomplex manifolds are  tori, $K3$ surfaces and  quaternionic Hopf surfaces. Such a classification in dimension $4n$, $n\geq 2$,  is far from being complete. Left invariant hypercomplex structures on compact Lie groups were first constructed by Spindel et al.  \cite{SSTVP} from the point of view of supersymmetry. Later, Joyce \cite{joy} gave a different proof of their result and considered the case of general homogeneous spaces. Regarding the non-compact case, Dotti and Fino studied in \cite{DF1} the existence of left invariant hypercomplex structures on nilpotent Lie groups; in particular, they gave the classification of such groups in dimension 8, proving that they are all at most 2-step nilpotent. 
In dimension 4, there is a unique hypercomplex almost abelian Lie algebra, namely, $\g=\R e_0\ltimes_A \R^{3}$ where $A$ is the identity matrix (see \cite{Bar}).

In \cite{AB} we started the study of left invariant hypercomplex structures on almost abelian Lie groups and their associated solvmanifolds.  We gave a characterization of  the corresponding hypercomplex  almost abelian Lie algebras and   using this characterization we were able to classify all 8-dimensional hypercomplex almost abelian Lie groups. We found that, in this family,  there are infinitely many pairwise non-isomorphic Lie groups. We also provided  examples of hypercomplex almost abelian Lie groups in any dimension $4n$, $n\geq 2$. Furthermore, we  investigated the existence of lattices in these groups and exhibited several  hypercomplex solvmanifolds. 

In this article we continue the study of hypercomplex almost abelian Lie groups. We classify such groups in the nilpotent case for arbitrary dimensions (see Theorem \ref{thm:classif_nilp} and Corollary \ref{cor:param-nilp_hcx}). In the 12-dimensional case, we obtain the complete classification of hypercomplex almost abelian Lie groups (Theorems \ref{thm: classification1} and \ref{thm: classification2}). The proof of the  classification theorems relies on the analogue of the Jordan normal form of quaternionic matrices \cite[Theorem 5.5.3]{Rod}. In the nilpotent case, all simply connected almost abelian Lie groups admit lattices due to Malcev's well-known criterion \cite{Mal}.  In \S\ref{sec:lattices}, we determine which 12-dimensional  hypercomplex almost abelian Lie groups admit lattices. 

Lastly, for any $n>1$, we provide a method to construct a $(4n+4)$-dimensional hypercomplex almost abelian solvmanifold beginning with a polynomial of degree $n$ in a distinguished family $\Delta_n\subset \Z[x]$. Indeed, given $p\in \Delta_n$ we can determine a matrix $A_p\in \operatorname{SL}(4n+3,\R)$ which gives rise to a hypercomplex almost abelian Lie group $G_p$ of completely solvable type, together with a lattice $\Gamma_p\subset G_p$ (Proposition \ref{prop:delta_n}). Moreover, we show that the map which associates to each $p\in \Delta_n$  the solvmanifold $\Gamma_p\backslash G_p$ is not one-to-one, but it is in general two-to-one (Theorem \ref{thm:p_or_p*}).  In \S\ref{subsec:deltan} we prove  several properties of  the family $\Delta_n$. In particular, we show that $\Delta_n$ is infinite (Lemma \ref{lem:infinite}).  

Appendix A contains an analogue of Theorem \ref{thm:classif_nilp} for 
 almost abelian Lie algebras admitting a complex structure (Theorem \ref{thm:classif_nilp-complex}). In Appendix B we determine all the Lie group isomorphisms between simply connected almost abelian Lie groups, which are needed in \S\ref{sec:polyn}.


\section{Preliminaries}\label{S: prelim}
\subsection{Hypercomplex manifolds} A complex structure on a differentiable manifold $M$ is an automorphism $J$ of the tangent bundle $TM$ satisfying $J^2=-I $  and the 
integrability condition $N_{J}(X,Y) =0$ for all  vector fields $X,Y$ on $M,$ where $N_J$ is the Nijenhuis tensor:
\begin{equation}  N_{J}(X,Y) =
[X,Y]+J([JX,Y]+[X,JY])-[JX,JY].  \label{nijen} \end{equation}
Recall that the integrability of $J$ is equivalent to the existence of an atlas  on $M$ such that the transition functions are holomorphic maps \cite{new}. 

A hypercomplex structure on $M$ is a triple of complex structures $\hcx$, $\al=1,2,3$, on $M$ satisfying the following conditions: 
\begin{equation}    \label{quat}
J_1J_2=-J_2J_1=J_3, \end{equation}
It then follows that $M$ has a family of complex structures 
$J_{y}=y_1J_1+y_2J_2+y_3J_3$ parameterized by points $y = (y_1,y_2,y_3)$ in the unit sphere $S^2 \subset \R^3$. It follows from \eqref{quat} that $T_pM$, for each $p\in M$, has an $\H$-module structure, where $\H$ denotes the quaternions; in particular, $\dim M \equiv 0 \; \pmod{4}$.

Given a hypercomplex structure $\hcx$ on $M$, there is a unique torsion-free connection $\nabla$ on $M$ such that $\nabla J_\al=0, \; \al =1,2,3$. It is called the Obata connection (see \cite{Ob,Sol}). The holonomy group of the Obata connection, $\operatorname{Hol}(\nabla)$, is contained in the quaternionic  general linear group $\operatorname{GL}(n, \H)$ (see \eqref{eq:GL_kH} below). The group $\operatorname{Hol}(\nabla)$ is an important invariant of  hypercomplex manifolds, see for instance \cite{IP,SV}. 

A hyperhermitian structure on $M$ is a pair $(\hcx ,g )$ where $\hcx$ is a hypercomplex structure and $(J_{\al},g )$ is Hermitian for $\al =1,2,3$.  
An interesting subclass of hyperhermitian structures is given by hyper-K\"ahler structures \cite{Cal}, which are hyperhermitian structures such that $(J_\al , g)$ is K\"ahler for $\al =1,2,3$, that is, the  K\"ahler forms $\omega_\al$ associated to $(J_\al,g)$ are closed, $\al =1,2,3$. In this case, the Levi-Civita connection coincides with the Obata connection and its holonomy group is contained in $\operatorname{Sp}\,(n)$, where $\dim M=4n$.  Since $\operatorname{Sp}\,(n)\subset \operatorname{SU}\,(2n)$,  hyper-K\"ahler metrics are Ricci-flat. 
A less restrictive class of hyperhermitian structures are 
the so-called hyper-K\"ahler with torsion (or HKT) structures \cite{HP}. These are hyperhermitian structures satisfying $\partial\Omega=0$, where $\Omega=\omega_2+i\omega_3$ and $\partial$ is the Dolbeault differential on $(M,J_1)$, with  $\omega_\al$ as above. 
We note that the class of hyper-K\"ahler manifolds is strictly contained in the class of HKT manifolds (see \cite{DF}) and this, in turn, is a proper class of hypercomplex manifolds (see \cite{FG}). Some recent contributions to the subject can be found in \cite{BGV,GeTa,GLV}, among many others. 

\subsection{Almost abelian solvmanifolds} A \textit{solvmanifold} is a compact quotient $\Gamma\backslash G$, where $G$ is a simply connected solvable Lie group and $\Gamma$ is a discrete subgroup of $G$. Such a subgroup $\Gamma$ is called a \textit{lattice} of $G$. When $G$ is nilpotent and $\Gamma\subset G$ is a lattice, the compact quotient $\Gamma\backslash G$ is known as a nilmanifold.

It follows that $\pi_1(\Gamma\backslash G)\cong \Gamma$ and  $\pi_n(\Gamma\backslash G)=0$ for $n>1$. Furthermore, solvmanifolds are determined up to diffeomorphism by their fundamental groups. In fact:

\begin{theorem}\cite{Mos}\label{thm:solv-isom}
If $\Gamma_1$ and $\Gamma_2$ are lattices in simply connected solvable Lie groups 
$G_1$ and $G_2$, respectively, and $\Gamma_1$ is isomorphic to $\Gamma_2$, then $\Gamma_1 \backslash G_1$ is diffeomorphic to $\Gamma_2 \backslash G_2$.
\end{theorem}

A solvable Lie group $G$ is called completely solvable if the adjoint operators $\ad_x:\g\to\g$, with $x\in \g=\operatorname{Lie}(G)$, have only real eigenvalues. 
The conclusion of the previous theorem can be strengthened when both solvable Lie groups $G_1$ and $G_2$ are completely solvable. Indeed, this is the content of Saito's rigidity theorem:

\begin{theorem}\cite{Sai}\label{thm:Saito}
Let $G_1$ and $G_2$ be simply connected completely solvable Lie groups and $\Gamma_1 \subset G_1, \, \Gamma_2\subset G_2$ lattices. Then every isomorphism
$f: \Gamma_1 \to \Gamma_2$ 
extends uniquely to an isomorphism of Lie groups $F: G_1 \to G_2$.
\end{theorem}


Let $G$ be a Lie group with Lie algebra $\g$. A hypercomplex structure $\hcx$ on $G$ is said to be left invariant if left translations by elements of $G$ are hyperholomorphic, i.e. holomorphic with respect to $J_\al$ for $\al=1,2,3$. In this case $\hcx$ is determined by the value at the identity of $G$, which corresponds to a hypercomplex structure on $\g$. We point out that if $\Gamma $ is a lattice in $G$, any left invariant hypercomplex structure on $G$ induces a hypercomplex structure on $\Gamma \backslash G$ which is called invariant. In this case, the natural projection $G \to \Gamma \backslash G$ is hyperholomorphic. 

We recall next that a  Lie group $G$ is said to be {\em almost abelian} if its Lie algebra $\g$ has a codimension one abelian ideal. Such a Lie algebra will be called almost abelian, and it can be written as $\g= \R e_0 \ltimes \mathfrak{u}$, where $\mathfrak u$ is an abelian ideal of $\g$, and $\R$ is generated by some $e_0\in \g$. After choosing a basis of $\u$, we may identify $\u$ with an abelian Lie algebra $\R^{d}$ and we may write $\g=\R e_0\ltimes_A \R^{d}$ for some $A\in \mathfrak{gl}(d,\R)$. 

Accordingly, the Lie group $G$ can be written as a semidirect product $G=\R\ltimes_\varphi \R^{d}$, where the action is given by $\varphi(t)=e^{t A}$. We point out that a non-abelian almost abelian Lie group is $2$-step solvable, and it is nilpotent if and only if the operator $A$ is nilpotent.

Regarding the isomorphism classes of almost abelian Lie algebras we have the following result, proved in \cite{Fr}.

\begin{lemma}\label{lem:ad-conjugated}
Two almost abelian Lie algebras $\g_1=\R e_1 \ltimes_{A_1} \R^d$ and $\g_2=\R e_2 \ltimes_{A_2}\R^d$ are isomorphic if and only if there exists $c\neq 0$ such that $A_2$ and $cA_1$ are conjugate. 
\end{lemma}

\begin{remark}\label{rem:nilp-sim}
It follows that two nilpotent almost abelian  Lie algebras as above are isomorphic if and only if $A_1$ and $A_2$ are conjugate, since for   any   nilpotent matrix $N$,   $cN$ and $N$ are conjugate  whenever $c\neq 0$.     
\end{remark}


In general, it is not easy to determine whether a given Lie group $G$ admits a lattice. A well known restriction is that if this is the case then $G$ must be unimodular (\cite{Mi}), i.e. the Haar measure on $G$ is left and right invariant, which is equivalent, when $G$ is connected, to  $\tr(\ad_x)=0$ for any $x$ in the Lie algebra $\g$ of $G$. In the nilpotent case there is a well-known criterion  due to Malcev:

\begin{theorem}\cite{Mal}\label{thm:Mal}
A simply connected nilpotent Lie group has a lattice if and only if its Lie algebra admits a basis with respect to which the structure constants are rational.
\end{theorem}

On the other hand, there is a criterion for existence of lattices on almost abelian Lie groups  which will prove very useful in forthcoming sections:

\begin{proposition}\label{prop: Bock}\cite{Bo}
Let $G=\R\ltimes_\varphi\R^d$ be a unimodular almost abelian Lie group. Then $G$ admits a lattice if and only if there exists  $t_0\neq 0$ such that $\varphi(t_0)$ is conjugate to a matrix in $\operatorname{SL}(d,\Z)$. In this situation, a lattice is given by $\Gamma=t_0 \Z\ltimes P\mathbb Z^{d}$, where $P\in \operatorname{GL}(d,\R)$ satisfies $P^{-1}\varphi(t_0)P\in \operatorname{SL}(d,\Z)$. 
\end{proposition}

Note that if $E:=P^{-1}\varphi(t_0)P$ then $\Gamma\cong \Z\ltimes_E \Z^d$, where the group multiplication in this last group is given by 
\[ (m,(p_1,\ldots,p_d))\cdot (n,(q_1,\ldots, q_d))=(m+n,(p_1,\ldots,p_d)+E^m(q_1,\ldots, q_d)). \]


\subsection{Quaternionic linear algebra}\label{sec:quat_linear}
Given a hypercomplex structure $\hcx$ on $\R^{4q}$, we will denote by
\begin{equation}\label{eq:GL_kH} \operatorname{GL}(q, \H):= \{T\in \operatorname{GL}(4q,\R) : TJ_\al =J_\al T \text{ for all } \al \}
\end{equation}
the quaternionic general linear group, with corresponding Lie algebra:
\begin{equation}   \label{eq:gl_kH} 
 \mathfrak{gl}(q,\mathbb{H})= \{T\in \mathfrak{gl}(4q,\R) : TJ_\al =J_\al T \text{ for all } \al \}. \end{equation}

In view of Theorem \ref{thm:characterization} below, we need to understand better the set of matrices  $\g\l (q,\H)\subset \g\l (4q,\R)$. In order to do so, we will use the fact that there is an $\R$-algebra isomorphism between 
$\g\l (q,\H)$ and the set of  $q\times q$ matrices with entries in $\H$, denoted by $M_q(\H)$, and then we will apply results from~\cite{Rod}.

The set $\H^q$ of column vectors with $q$ quaternionic components  will be considered as a right vector space over $\H$. On the other hand, the set $M_q(\H)$ will be considered as a left vector space over $\H$. These conventions allow us to interpret $Q\in M_q(\H)$ as a linear transformation $Q: \H^q \to \H^q$, which acts by the standard matrix-vector multiplication.

We can identify $\mathfrak{gl}(q,\H)\subset \mathfrak{gl}(4q,\R)$ with $M_q(\H)$ via the $\R$-linear map $\sigma: \mathfrak{gl}(q,\mathbb{H}) \to M_q(\H)$ defined by:
\begin{equation}\label{eq:BH}
    \sigma(B)= X+iY-jZ+kW,
\end{equation}
where $B$ is given as in \eqref{eq:matrixB}. 
Note that $\sigma$ is an isomorphism of unital $\R$-algebras with respect to matrix multiplication. In particular, $B\in \mathfrak{gl}(q,\mathbb{H})$ is nilpotent if and only if $\sigma (B)\in M_q(\H)$ is nilpotent. 


The quaternionic Jordan blocks are quaternionic matrices of the following form:
\begin{equation} \label{eq:J_lambda} J_m(\lambda)=\left[
\begin{array}{ccccc}
   \lambda  &  0 & \cdots & \cdots & 0 \\
   1 & \lambda & 0 & \cdots & 0 \\
   0 & 1 & \lambda & \cdots & 0 \\
   \vdots & \vdots & \ddots & \ddots & 0 \\
   0 & 0 & 0 & 1 & \lambda
\end{array}\right] \in M_m(\H), \quad \lambda \in \H.
\end{equation}
There is an analogue of the Jordan normal form for quaternionic matrices in terms of the quaternionic Jordan blocks (see, for instance, \cite[Section 5.5]{Rod} and references therein). Moreover, the scalars $\lambda$ on the diagonal can be chosen in $\C$, with $\operatorname{Im} \lambda \geq 0$ and with this choice the decomposition is unique up to permutation of the Jordan blocks.

Throughout the paper we will use the following notation: for  two matrices $M_1,M_2$, we denote  $M_1\oplus M_2=\begin{bmatrix}
    M_1&0 \\ 0&M_2
\end{bmatrix}$. We use a similar notation for 3 or more matrices. Moreover, 
given a  matrix $M$, $M^{\oplus p}$ denotes  $M\oplus \dots \oplus M$ ($p$ copies).

\begin{theorem}\label{thm:Jordan}
Given $Q\in M_q(\H)$ there exists an invertible $S\in M_q(\H)$ such that $S^{-1}QS$ has the form
\begin{equation}   \label{eq:Jordan} 
 S^{-1}QS= J_{m_1}(\lambda _1)\oplus \cdots \oplus J_{m_p}(\lambda _p),
\end{equation}
where $\lambda_l\in \C$ have non-negative imaginary parts. Furthermore, the expression above is unique up to an arbitrary permutation of blocks.
\end{theorem}

A proof of this theorem can be found in \cite[Theorem 5.5.3]{Rod}. The expression \eqref{eq:Jordan} is called the quaternionic Jordan form of $Q$. 

\begin{remark} \label{rem:nilp-deg} Assume that $m_1 \geq \dots \geq m_p$ in \eqref{eq:Jordan}. It follows that if $Q$ satisfies $Q^r=0$ and $Q^{r-1}\neq 0$ then $r=m_1\leq q$. 
\end{remark}

\section{Generalities on hypercomplex almost abelian Lie algebras
}
In this section we recall from \cite{AB} the characterization of almost abelian Lie algebras admitting a hypercomplex structure (see also \cite{MeTa}).  We also prove some results to refine Lemma \ref{lem:ad-conjugated}  in our particular case.

\begin{theorem}\label{thm:characterization} \cite[Theorem 3.2]{AB}
Let $\g$ be a $4n$-dimensional almost abelian Lie algebra with codimension one abelian ideal $\u$,  admitting a hypercomplex structure $\hcx$. 
Let $\h := \u \cap J_1\u\cap J_2\u\cap J_3\u$ be the maximal $\hcx$-invariant subspace contained in $\u$. Then $\h$ is an abelian ideal of $\g$ and there exists a $\hcx$-invariant complementary 
subspace $\q= \operatorname{span}\{e_0,e_1,e_2,e_3\}$ of $\h$ with $e_{\al}=J_{\al}e_0$, such that $e_0\notin \u$, $e_\al \in \u$ and, moreover:
\begin{enumerate}
	\item[$\ri$] $[e_0,e_{\al}]=\mu e_{\al}+v_{\al}$ for some $\mu \in \R$ and $v_{\al}\in \h$, ${\al}=1,2,3$,
	\item[$\rii$] there exists $v_0\in \h$ such that $J_\al v_0=v_\al$ for all $\al$ and moreover,   $v_\al=J_\be v_\ga$ for any  cyclic permutation $(\al , \be , \ga)$  of $(1, 2, 3)$,
	\item[$\riii$] $[e_0,x]=Bx$ for any $x\in\h$, where $B\in\operatorname{End}(\h)$ satisfies $[B,J_\alpha|_\h]=0$, $\alpha=1,2,3$.
\end{enumerate}
\end{theorem}


It follows from Theorem \ref{thm:characterization}  that  $\g$ can be written as $\g=\R e_0\ltimes_A \R^{4n-1}$, where the matrix $A\in\mathfrak{gl}(4n-1,\R)$ defined by the adjoint action of $e_0$ on $\R^{4n-1}$ has the following expression in a basis $\{e_1, e_{2}, e_3\} \cup \mathcal B$ of $\u$,  where $\mathcal B$ is a basis of $\h$:
\begin{equation}\label{eq:matrixA} 
A=\left[
	\begin{array}{ccc|ccc}      
		\mu & & & & &\\
		& \mu & & & 0 & \\
		& & \mu & & & \\
		\hline
		| & | & |& & & \\
		v_1 & v_2 & v_3 & & B &\\
		| & | & | & & &
	\end{array}
	\right], \qquad B\in\mathfrak{gl}(n-1,\mathbb{H})\subset \mathfrak{gl}(4n-4,\R).
\end{equation} 

\begin{remark} It is clear that the converse of Theorem \ref{thm:characterization} holds, that is, given a matrix $A$ as in \eqref{eq:matrixA} the almost abelian Lie algebra $\R e_0\ltimes _A \R^{4n-1}$ admits a hypercomplex structure which extends a given one on $\R^{4n-4}$.     
\end{remark}

\medskip

By fixing a basis $\mathcal B$ of $\h$ of the form  $\mathcal B =\{f_j \}\cup \{ J_1f_j \}
\cup \{ J_2f_j \} \cup \{ J_3f_j \}$, for $1\leq j\leq n-1$, the matrix $B$ in \eqref{eq:matrixA} can be expressed as: 
\begin{equation}\label{eq:matrixB}
B= \begin{bmatrix} X & -Y & -Z & -W \\
Y & X & W & -Z \\
Z & -W & X & Y \\
W & Z & -Y & X
\end{bmatrix}, \qquad   X, Y,Z,W \in \mathfrak{gl} (n-1, \R ).
\end{equation}
In this basis, the operators $J_\alpha : \h \to \h$   take the following form:
\begin{equation*}
J_1=\begin{bmatrix}  & -I & & \\
I & & & \\
 & & & -I \\
 & & I & 
\end{bmatrix}, \qquad 
J_2=\begin{bmatrix}  &  & -I &  \\
 & & & I \\
 I & & &  \\
 & -I &  & 
\end{bmatrix}, \qquad
J_3=\begin{bmatrix} & &  & -I   \\
 & &  -I & \\
  & I & &  \\
I  &  &  & 
\end{bmatrix},
\end{equation*}
where $I$ is the $(n-1)\times (n-1)$ identity  matrix.

\medskip 

The following matrix will be used several times throughout the forthcoming sections:
\begin{equation}\label{eq:matrixU} U= \begin{bmatrix} 0&0&0\\
1&0&0\\
0&1&0\\
0&0&1    
\end{bmatrix}.
\end{equation}
It is related to $v_1,v_2,v_3$ in case $v_0\neq 0$. 

\begin{remark} \label{rem:dim_eigenspace}
We point out that, 
if $\mu$ is an eigenvalue of $B$, then $\ker(B-\mu I)^j$ is a $\hcx$-invariant subspace of~$\h$, for all $j$. In particular, $\dim \ker(B-\mu I)^j\, \equiv 0 \pmod 4$. If $v_0=0$ then \[ \dim \ker (A-\mu  I)^j=\dim \ker (B-\mu  I)^j +3\, \equiv 3 \pmod 4  \quad \text{ for all } j.\] 
If $v_0\neq 0$ there are two possibilities:
\begin{enumerate}
\item[$\ri$] $(B-\mu  I)^jv_0\neq 0$ for all $j$,
\item[$\rii$] there exists $\ell \geq 1$ such that $(B-\mu  I)^{\ell} v_0 =0$  and $(B-\mu  I)^{j} v_0 \neq 0$ for $1\leq j \leq\ell -1$. 
\end{enumerate}
Since $v_\al=J_\al v_0$, then $\{(B-\mu I)^jv_\al\,:\, \al=1,2,3 \}$ are linearly independent for all $j$ such that $(B-\mu  I)^jv_0\neq 0$. Therefore, in  case $\ri$ above, $\dim \ker (A-\mu  I)^j= \dim \ker (B-\mu  I)^j$ for all $j$. On the other hand, if $\rii$ holds, we have:  
\begin{equation}\label{eq:A'}
\dim \ker (A-\mu  I)^j=\begin{cases}   
\dim \ker (B-\mu  I)^j \, \equiv 0 \pmod 4, & \text{ for }  1\leq j\leq \ell  , \\
    \dim \ker (B-\mu  I)^j +3 \, \equiv 3 \pmod 4,  &  \text{ for }  j> \ell .
\end{cases}
\end{equation}     
Note that, for $\lambda \neq \mu$ (even for $\lambda\in \C$), either $\ker (A-\lambda  I)=0=\ker (B-\lambda  I)$ or  $\dim \ker (A-\lambda  I)^j=\dim \ker (B-\lambda  I)^j$ for all $j$. 
\end{remark}
 

We recall next another result from \cite{AB}, where it is shown that under certain assumptions we may perform a hypercomplex change of basis of $\g$ such that the vectors  $v_\al$ in Theorem ~\ref{thm:characterization} vanish for all $\al$. 

\begin{proposition}\label{prop:v=0}
With notation as in Theorem \ref{thm:characterization}, if $v_0 \in \operatorname{Im}(B-\mu I)$, then there exists a $\hcx$-invariant complementary 
subspace $\q '= \operatorname{span}\{e_0',e_1',e_2',e_3'\}$ of $\h$ with $e_{\al}'=J_{\al}e_0'$ such that $e_0'\notin \u$, $e_\al '\in\u$, $[e_0',e_\al']=\mu e_\al '$ and $\ad_{e_0'}|_\h =B$. In other words, the corresponding $v_0'\in \h$ vanishes.

In particular, this holds if $\mu$ is not an eigenvalue of $B$.
\end{proposition}


In case that $v_0 \notin \operatorname{Im}(B-\mu I)$, we may assume that it lies in a $\hcx$-invariant complementary subspace of $\operatorname{Im}(B-\mu I)$, as the next result shows.
\begin{lemma}\label{lem:vnot0} With notation as in Theorem \ref{thm:characterization}, if $v_0 \notin \operatorname{Im}(B-\mu I)$, let $W$ be a $\hcx$-invariant subspace of $\h$ such that 
$\h=\operatorname{Im}(B-\mu I)\oplus W.$ Then  there exists a $\hcx$-invariant complementary 
subspace $\q '= \operatorname{span}\{e_0',e_1',e_2',e_3'\}$ of $\h$ with $e_{\al}'=J_{\al}e_0'$ and 
$v_0 '\in W$ such that $e_0'\notin \u$, $e_\al '\in\u$, $[e_0',e_\al']=\mu e_\al '+J_\al v_0 '$ and $\ad_{e_0'}|_\h =B$. In other words, the corresponding $v_0 '\in \h$ lies in $W$. 
\end{lemma}
    
\begin{proof} 
Since $v_0 \notin \operatorname{Im}(B-\mu I)$, we can write 
\begin{equation}\label{eq:v0pr}       
v_0=(B-\mu I)x_0 +v_0' 
\end{equation}
with $0\neq v_0'\in W$. We set $e_0'= e_0-x_0$ and $e_\al '=J_\al e_0 '$.  Then 
\[ [e_0 ', e_\al ']= [e_0-x_0 , e_\al - J_\al x_0] = \mu e_\al +v_\al -B J_\al x_0. \]
Observe that \eqref{eq:v0pr} implies that $v_\al -BJ_\al x_0= -\mu J_\al x_0+ J_\al v_0 '$. Therefore, $[e_0 ', e_\al ']=\mu e_\al ' +J_\al v_0 '$ for all~$\al.$        
\end{proof}


Concerning the isomorphism classes of hypercomplex almost abelian Lie algebras, we can combine Lemma \ref{lem:ad-conjugated} and Theorem \ref{thm:characterization} to obtain the following result.

\begin{lemma}\label{lem:no-iso}
Let $\g =\R e_0 \ltimes _{A}\R^{4n-1}$ and $\g ' =\R e_0 \ltimes _{A'}\R^{4n-1}$ be isomorphic hypercomplex almost abelian Lie algebras such that $A, \, A'$ are as in \eqref{eq:matrixA} and let $c\in \R, \, c\neq 0$, such that $A$ is conjugate to $cA'$. Then:
 \begin{enumerate}
     \item[$\ri$] $\mu=c\mu'$,
     \item[$\rii$] $v_0=0$ if and only if $v_0'=0$, where $v_0, \, v_0'$ are defined as in Theorem \ref{thm:characterization}$\rii$,
    \item[$\riii$] $B$ is conjugate to $cB'$. 
 \end{enumerate}
\end{lemma}

\begin{proof}
Since  two real matrices are conjugate over $\R$ if and only if they are conjugate over $\C$,  it suffices to  show that $B$ and $cB'$ are conjugate over $\C$. 

Let $\lambda \in \R$ be an eigenvalue of $A$, $\lambda \neq \mu$. Then it is an eigenvalue of $B$ and its multiplicity is divisible by 4, while the multiplicity $m_A(\mu)$ of $\mu$ satisfies  $m_A(\mu)\equiv 3 \pmod{4}$.  Since $A$ is conjugate to $cA'$ and $c\mu '$ is the only real eigenvalue of $cA'$ with multiplicity $m_{cA '}(c\mu ')\equiv 3 \pmod{4}$, we must have $\mu=c\mu'$, and $\ri$ follows. 
    
We show next that $v_0= 0$ if and only if $v_0'= 0$. Indeed, assume that $v_0= 0$, then it follows from Remark \ref{rem:dim_eigenspace} that
\[\dim \ker (A-\mu I)=\dim \ker (B-\mu I) +3 \equiv 3 \pmod{4}.\]
If $v_0'\neq 0$, since $c\mu '=\mu$, we would have:
\[ \dim \ker (cA'-\mu  I)=\dim \ker (cB'-\mu I)  \equiv 0 \pmod{4},\]
and since $\dim \ker (A-\mu I)=\dim \ker (cA'-\mu I)$ we get a contradiction, hence $v_0 '= 0$. Similarly,  $v_0'= 0$ implies $v_0 = 0$,   therefore, $\rii$ holds.

Let $\lambda '$ be an eigenvalue of $A'$, then $c\lambda '$ is an eigenvalue of $A$ and $\riii$ will follow if we show  that $\dim \ker (B'-\lambda ' I)^j=\dim \ker (B-c\lambda ' I)^j$ for all $j$. We consider separately the cases $\lambda'=\mu'$ and $\lambda'\neq  \mu'$. 

If $\lambda '=\mu '$, then $\ri$ holds, so $\mu=c \mu '$. We compute $\dim \ker (A'-\mu ' I)^j$ and     $\dim \ker (A-\mu  I)^j$ as in Remark \ref{rem:dim_eigenspace}. Since 
$\dim \ker (A-\mu  I)^j=\dim \ker (cA'-c\mu'  I)^j =\dim \ker (A'-\mu '  I)^j $ for all $j$, we conclude that 
\[
\dim \ker (B-\mu I)^j=\dim \ker (B'-\mu ' I)^j, \quad \text{ for all } j,  
\]
that is, 
\[
\dim \ker (B-c\mu ' I)^j=\dim \ker (B'-\mu ' I)^j, \quad \text{ for all } j. 
\]
If $\lambda '\neq\mu'$, then  $c \lambda' \neq \mu $ and it follows from Remark \ref{rem:dim_eigenspace} that
\begin{eqnarray*}
 \dim \ker (B'-\lambda ' I)^j&= &  \dim \ker (A'-\lambda ' I)^j=  \dim \ker (cA'-c\lambda ' I)^j \\
  &=&  \dim \ker (A-c\lambda ' I)^j = \dim \ker (B-c\lambda ' I)^j .
\end{eqnarray*}
 This concludes the  proof of $\riii$. 
\end{proof}


\section{Nilpotent hypercomplex almost abelian Lie algebras}

In this section we determine all the isomorphism classes of nilpotent hypercomplex almost abelian Lie algebras  by applying Theorem \ref{thm:Jordan} in the nilpotent setting. 

Our first result establishes a relation between the nilpotency degree and the dimension of the nilpotent Lie algebra. 

\begin{lemma}\label{lem:nilp-no-hcx}
Let $\g=\R e_0\ltimes \R ^{4n-1}$ be a $k$-step nilpotent  $4n$-dimensional almost abelian Lie algebra. If $\g$ admits a hypercomplex structure then $k\leq n$. \end{lemma}
\begin{proof}
Let $\u \cong \R ^{4n-1}$ be the codimension-one abelian ideal of $\g$. It follows from Theorem \ref{thm:characterization} that 
the matrix $A$ of $\ad_{e_0}|_\u$ in some basis of $\u$ takes the  form \eqref{eq:matrixA}   with $\mu =0$. Since  $\g$ is $k$-step nilpotent,  $A^k=0$ and $A^{k-1}\neq 0$, which implies that $B^k=0$. 

If $B^{k-1}\neq 0$ then $\sigma (B)\in M _{n-1} (\H )$ satisfies $(\sigma(B))^{k} =0$ and $(\sigma(B))^{k-1}\neq 0$,  hence $k\leq n-1$ (see Remark \ref{rem:nilp-deg}). 

If $B^{k-1}= 0$, then $B^{k-2}\neq  0$ since $A^{k-1}\neq 0$. Therefore, $(\sigma(B))^{k-1} =0$ and $(\sigma(B))^{k-2}\neq 0$, which implies that $k-1\leq n-1$, and the lemma follows. 
\end{proof}

\begin{remark} Let $\g=\R e_0\ltimes_{A} \R^{4n-1}$  be a $k$-step nilpotent hypercomplex  almost abelian Lie algebra, with $A$ as in \eqref{eq:matrixA} with $\mu=0$.   We observe that the minimal polynomial of the matrix $B$ is either  $x^k$ or $x^{k-1}$.     
\end{remark}

In order to study the nilpotent hypercomplex almost abelian  Lie algebras, let $\g=\R \ltimes  _A \R ^{4n-1}$ where $A$ is  as in \eqref{eq:matrixA} with $\mu=0$ and $B\in \g\l (n-1, \H)$, $B$ nilpotent. If $B\neq 0$,  
 the corresponding Jordan blocks in the quaternionic Jordan form of $\sigma (B)$ are given as in Theorem \eqref{thm:Jordan} with $\lambda =0$.  We set $ J_m := J_m(0)$ to simplify notation. It follows that there exists an invertible $S\in M_{n-1}(\H)$ such that $S^{-1}\sigma(B)\, S$ has the form
\begin{equation}   \label{eq:Jordan_nilp} 
 S^{-1}\sigma(B)\, S= ({ J_{m_1}})^{\oplus p_1}\oplus \cdots \oplus  ( J_{m_r})^{\oplus p_r}\oplus { 0}_s, \quad m_1 > \dots >m_r\geq 2, \;s\geq 0, \; p_k>0,
\end{equation}
for all $k$, where $0_s$ is the zero $s\times s$ matrix. 
We will encode all this data associated to $B$ using the following notation:
\begin{equation}\label{eq:Sigma} \Sigma(B):=(r,m_1,\ldots, m_r, p_1,\ldots, p_r, s),  \qquad r \geq 1,  \;\; m_1 > \dots >m_r\geq 2, \;s\geq 0, \; p_k>0 .  
\end{equation}
Note that $\Sigma(B)$ is well defined due to the uniqueness of the quaternionic Jordan form, fixing the order of the blocks in decreasing size. We point out that \begin{equation}\label{eq:dimn}
    n-1=\sum_{i=1}^r m_i p_i +s .
\end{equation}
If $B=0$ we have $r=0, \; s=n-1$.

Note that the matrix $B$ such that $\sigma(B)$ satisfies \eqref{eq:Jordan_nilp} has minimal polynomial $x^{m_1}$. If $B=0$ then, since $\g$ is not abelian, we must have $v_0\neq 0$ and  $A$ has minimal polynomial $x^2$. 


We prove next a general result relating the real conjugacy class of a nilpotent matrix $B$ in $\mathfrak{gl}(q,\mathbb{H})$  with the quaternionic    conjugacy class  of  $\sigma(B)$.

\begin{proposition}\label{prop:tuples}
Let $B_1$ and $B_2$ be non-zero nilpotent matrices in $\mathfrak{gl}(q,\mathbb{H})\subset \mathfrak{gl}(4q,\mathbb{R})$. If 
$B_1$ and $B_2$ are conjugate in $\mathfrak{gl}(4q,\mathbb{R})$ then 
$\sigma(B_1)$ and $\sigma(B_2)$ are conjugate  in $M_q(\H)$. Equivalently, $\Sigma(B_1)=\Sigma(B_2)$. 
\end{proposition}

\begin{proof} Let $\Sigma(B_1)=(r,m_1,\ldots, m_r, p_1,\ldots, p_r, s), \; \Sigma(B_2)=(r',m'_1,\ldots, m'_{r'}, p'_1,\ldots, p'_{r'}, s')$. We point out that $\sigma^{-1}(J _m)=j_m ^{\oplus 4}$, where $j_m$ is the elementary Jordan block
\begin{equation} \label{eq:matrixjm} j_m=\left[
\begin{array}{ccccc}
   0  &  0 & \cdots & \cdots & 0 \\
   1 & 0 & 0 & \cdots & 0 \\
   0 & 1 & 0 & \cdots & 0 \\
   \vdots & \vdots & \ddots & \ddots & 0 \\
   0 & 0 & 0 & 1 & 0
\end{array}\right] \in M_m(\R).
\end{equation}
 Therefore, the real Jordan forms of $B_1$ and $B_2$ are given by 
 \[ 
({ j_{m_1}})^{\oplus 4p_1}\oplus \cdots \oplus  ( j_{m_r})^{\oplus 4p_r}\oplus { 0}_{4s} \quad \text{and} \quad
({ j_{m'_1}})^{\oplus 4p'_1}\oplus \cdots \oplus  ( j_{m'_{r'}})^{\oplus 4p'_{r'}}\oplus { 0}_{4s'},
 \]
 respectively, hence, $\Sigma(B_1)=\Sigma(B_2)$, as asserted. 
\end{proof}


For each $n\geq 2$, let $\Sigma _{n-1}$ denote the set of all possible tuples
\[ (r,m_1,\ldots, m_r, p_1,\ldots, p_r, s), \quad \;\; r>0,\, m_1 > \dots >m_r\geq 2, \;s\geq 0, \; p_k>0 , \] satisfying \eqref{eq:dimn}. Consider the set of non-zero nilpotent matrices  $ N_{n-1}(\H)\subset \g\l (n-1, \H)\subset \g\l (4n-4, \R)$. The following corollary is a straightforward consequence of Proposition \ref{prop:tuples}.

\begin{corollary}\label{cor:tuples}
The conjugacy classes in $\g\l (4n-4, \R)$ of elements in $N_{n-1}(\H)$ are parametrized by $\Sigma _{n-1}$.
\end{corollary}

In what follows we will need to work with the  nilpotent matrices introduced below, which are conjugate to $j_m^{\oplus 4}$ by a hypercomplex change of basis:
\begin{equation} \label{eq:matrixJm} \J_m=\left[
\begin{array}{ccccc}
   0_4  &  0_4 & \cdots & \cdots & 0_4 \\
   I_4 & 0_4 & 0_4 & \cdots & 0_4 \\
   0_4 & I_4 & 0_4 & \cdots & 0_4 \\
   \vdots & \vdots & \ddots & \ddots & 0_4 \\
   0_4 & 0_4 & 0_4 & I_4 & 0_4
\end{array}\right] \in M_{4m}(\R),
\end{equation}
where $0_4$ and $I_4$ are the $4\times 4$ zero and identity matrices, respectively.
 Note that $\R^{4m}$ decomposes as 
 \begin{equation} \label{eq:Wm}    
 \R^{4m}=\operatorname{Im} (\J_m) \oplus W_m , 
 \end{equation}
 where both subspaces are $\hcx$-invariant and $\dim W_m =4$.

 Let $B \in \g\l (q,\H)$ be any nilpotent matrix, then $B$  is conjugate to:
\begin{equation}\label{eq:Bnilp}
     ({ \J _{m_1}})^{\oplus p_1}\oplus \cdots \oplus  ( \J _{m_r})^{\oplus p_r} 
\oplus 0_{4s}, 
\end{equation}
where $\Sigma(B)=(r,m_1,\ldots, m_r, p_1,\ldots, p_r, s)$ as in \eqref{eq:Sigma}. Let $\widetilde{W}_{r+1}$ be a $\hcx$-invariant complementary subspace of $\operatorname{Im}B \cap \ker B$ in $\ker B$, so that $\dim \widetilde{W}_{r+1}=4s$. Using \eqref{eq:Wm} we obtain:
\[
(\R^{4m_l})^{\oplus p_l}= (\operatorname{Im} (\J_{m_l}) \oplus W_{m_l})^{\oplus p_l}= (\operatorname{Im} (\J_{m_l}))^{\oplus p_l}\oplus (W_{m_l})^{\oplus p_l}.
\]
Setting $\widetilde{W}_l:= (W_{m_l})^{\oplus p_l}$, $1\leq l \leq r$, we have:
\begin{equation}\label{eq:W} W= \widetilde{W}_{1}\oplus \cdots \oplus \widetilde{W}_{r+1}\end{equation} 
and it follows that 
\[
\h =\operatorname{Im} B \oplus W,
\]
where $\operatorname{Im} B$ and $W$ are $\hcx$-invariant.

In order to state Theorem \ref{thm:classif_nilp} below, we introduce first some notation. Consider the following nilpotent matrices: 

\begin{equation}\label{eq:matrixN}
      N_\ell= \left[ \begin{array}{c|ccc}
    0_3&  & & \\
    \hline 
 \begin{array}{c}
    U  \\ 
    0_{k \times   3}
    \end{array} & & \J_{m_\ell} &  
\end{array}\right] \oplus (\J _{m_\ell})^{\oplus (p_\ell-1)}, \; 1\leq \ell\leq r, \quad 
N= \left[ \begin{array}{c|c}
    0_3&    \\
    \hline 
\begin{array}{c}
    U  \\ 
    0_{k\times 3}
    \end{array}
  &  0_{4s} 
\end{array}\right], \; s>0,
\end{equation}
where $U$ is the $4\times 3$ matrix defined in \eqref{eq:matrixU}, $\J_{m_\ell}$ is  as in \eqref{eq:matrixJm} and   $0_{k \times 3}$ is the zero  $k\times 3$ matrix with $k =4(m_\ell -1)$ in the matrix $N_\ell$ and $k=4(s-1)$ in the matrix $N$. 
 Let
\begin{equation}
     \label{eq:Al} \begin{split} 
     A_0&=0_3 \oplus ({ \J _{m_1}})^{\oplus p_1}\oplus \cdots \oplus  ( \J _{m_r})^{\oplus p_r} 
\oplus 0_{4s}, \\
     A_\ell&= \bigoplus_{i=1}^{\ell-1}({ \J_{m_i}})^{\oplus p_i}\oplus   \; N_\ell \; \oplus\bigoplus_{i=\ell+1}^{r} ( \J _{m_{i}})^{\oplus p_{i}}\oplus  { 0}_{4s}, \quad 1\leq \ell\leq r,\\
  A_{r+1}&= ({ \J _{m_1}})^{\oplus p_1}\oplus \cdots \oplus  ( \J _{m_r})^{\oplus p_r}\oplus N.
\end{split} \end{equation}
We point out that $A_{r+1}$ is only defined when $s>0$. 


We prove next the main theorem of this section. For each nilpotent matrix $B\in \g\l(n-1, \H)\subset \g\l(4n-4,\R)$ we will determine all the corresponding matrices $A$ as in \eqref{eq:matrixA} with $\mu =0$ which give rise to pairwise non-isomorphic  nilpotent hypercomplex almost abelian Lie algebras $\g=\R e_{0} \ltimes _A \R^{4n-1}$.
More precisely, if  $B\neq 0$ with corresponding $\Sigma (B)=(r,m_1,\ldots, m_r, p_1,\ldots, p_r, s)$ as in \eqref{eq:Sigma},  then $B$ gives rise to  exactly $r+2-\delta_{s,0}$ isomorphism classes of $4n$-dimensional hypercomplex nilpotent  almost abelian Lie algebras, where $\sum_{i=1}^r m_i p_i +s=n-1$ and  $\delta_{s,0}$ denotes the Kronecker delta. On the other hand, if $B=0$ there is only one isomorphism class. 

\begin{theorem}\label{thm:classif_nilp} Let 
$B\in  \g\l(n-1, \H)\subset \g\l(4n-4,\R)$ be a nilpotent matrix and consider $A$ as in \eqref{eq:matrixA}, for some $v_\al$, with $\mu=0$, where  $B$ is  the given matrix.  Consider the hypercomplex almost abelian Lie algebra $\g_A=\R e_0 \ltimes_A \R^{4n-1}$. 
\begin{enumerate}
    \item[$\ri$] If $B=0$ then $\g _A$ is $2$-step nilpotent and isomorphic to  $\g _N$, where $N$ is as in \eqref{eq:matrixN} with $s=n-1$.
\item[$\rii$] If $B\neq 0$ with
\[
\Sigma (B)=(r,m_1,\ldots, m_r, p_1,\ldots, p_r, s), \quad r\geq 1,\; m_1 > \dots >m_r\geq 2, \;s\geq 0, \; p_k>0,
\]
 then there exists a unique integer $\ell$ with  $0\leq \ell \leq r+1-\delta_{s,0}$  such that $\g _A$ is isomorphic to $\g _{A_\ell}$, where $A_0$ and $A_\ell$, $1\leq \ell\leq r+1$, are defined in  \eqref{eq:Al}. The Lie algebra  $\g_{A_1}$ is $(m_1+1)$-step nilpotent and $\g_{A_\ell}$ is $m_1$-step nilpotent for $\ell\neq 1 $.
 \end{enumerate}
 Moreover, the Lie algebras $\g_N, \g_{A_0}, \dots , \g_{A_{r+1}}$ are pairwise non-isomorphic. 
    \end{theorem}
\begin{proof} 
$\ri$ If $B=0$, then the matrix $A$  has rank $3$ and $A^2=0$. Therefore,
\[ \dim \ker A= 4s =\dim \ker N,
\]
where $N^2=0$. This implies that $A$ is conjugate to $N$, that is, $\g _A$ is isomorphic to $\g _N$,  which is 2-step nilpotent. 


$\rii$ Assume now $B\neq 0$, so that $r\geq 1$. If $v_0=0$ the isomorphism class of  $\g_A$  is completely determined by the integers $r, s$, $m_1, \dots , m_r$ and $p_1, \dots , p_r$ from \eqref{eq:Jordan_nilp}. In other words, $\g_A$ is isomorphic to $\R e_0\ltimes_{A_0} \R ^{4n-1}$ where $A_0$ is as in \eqref{eq:Al}, and this Lie algebra is $m_1$-step nilpotent.

Consider now $v_0\neq 0$.  According to Lemma \ref{lem:vnot0} we may assume that $v_0\in W$, where $W$ is the subspace  defined in \eqref{eq:W}. We decompose 
\[ v_0=v_0^1+ \cdots + v_0^{r+1}, \qquad v_0^l\in \widetilde{W}_l,\] 
and let \[
l_0:= \min \{ \, l \, :\,  v_0^l \neq 0 \} \geq 1.
\]
Note that $v_0^{r+1}=0$ when $s=0$. We will show next that  $A$ is conjugate to $A_{l_0}$, that is, $\g_A$ is isomorphic to $\g _{A_{l_0}}$. We consider three cases.


\textsl{Case (1):} Assume first that $l_0=1$. Note that in this case the minimal polynomial of $A$ is $x^{m_1 +1}$. We compute $\dim \ker (A^j)$, $1\leq j\leq m_1$:
\begin{equation}\label{eq:caso(1)}
\dim \ker (A^j)=\begin{cases}\displaystyle{  4j\sum_{i=1}^{r}  p_i + 4s,} & 1\leq j\leq m_r, \\ \vspace*{-.3cm}
\\
 \displaystyle{ 4j\sum_{i=1}^{k-1}  p_i +4 \sum_{i=k}^{r}  m_i p_i + 4s,} & m_k <j\leq m_{k-1} , \;\; 2\leq k\leq r
 .\end{cases}
    \end{equation}
It follows that $\dim \ker (A^j)= \dim \ker (A_1^j) $ for $1\leq j\leq m_1$, therefore, the theorem holds for $l_0=1$. 


\textsl{Case (2):} Assume next that $2\leq l_0 \leq r$. In this case, the minimal polynomial of both $A$ and $A_{l_0}$ is $x^{m_1}$, so  we compute $\dim \ker (A^j)$, $1\leq j\leq m_1-1$: 
\begin{equation}\label{eq:caso(2)}
\dim \ker (A^j)=\begin{cases}    
 \displaystyle{4j\sum_{i=1}^{r}  p_i + 4s,} & 1\leq j\leq m_{r}, \\
\vspace*{-.3cm}
 \\
\displaystyle{ 4j\sum_{i=1}^{k-1}  p_i +4 \sum_{i=k}^{r}  m_i p_i + 4s,} &  m_k <j\leq m_{k-1},\;\;  l_0+1\leq k\leq r,
 \\
\vspace*{-.3cm}
 \\
\displaystyle{ 4j\sum_{i=1}^{k-1}  p_i +4 \sum_{i=k}^{r}  m_i p_i + 4s+3,} &  m_k <j\leq m_{k-1},\;\; 3\leq k\leq l_0, 
\\ \vspace*{-.3cm}
\\
 \displaystyle{ 4j  p_1 +4 \sum_{i=2}^{r}  m_i p_i + 4s+3,} & m_2 <j\leq m_{1}-1 .
\end{cases}
\end{equation}
We have that $\dim \ker (A^j)= \dim \ker (A_{l_0}^j) $ for $1\leq j\leq m_1-1$, therefore, $A$ is conjugate to $A_{l_0}$ for $2\leq l_0\leq r$. 


 \textsl{Case (3):} Assume next  $l_0=r+1$, which implies that $s>0$.   In this case, the minimal polynomial of both $A$ and $A_{r+1}$ is $x^{m_1}$, so  we compute $\dim \ker (A^j)$, $1\leq j\leq m_1-1$: 
\begin{equation}\label{eq:caso(3)}
\dim \ker (A^j)=\begin{cases}    
 \displaystyle{4\sum_{i=1}^{r}  p_i + 4s,} & j=1, \\
\vspace*{-.3cm}
 \\ \displaystyle{4j\sum_{i=1}^{r}  p_i + 4s+3,} & 2\leq j\leq m_{r}, \\
\vspace*{-.3cm}
 \\
\displaystyle{ 4j\sum_{i=1}^{k-1}  p_i +4 \sum_{i=k}^{r}  m_i p_i + 4s +3,} &  m_k <j\leq m_{k-1},\;\;  3\leq k\leq r,
\\ \vspace*{-.3cm}
\\
 \displaystyle{ 4j  p_1 +4 \sum_{i=2}^{r}  m_i p_i + 4s+3,} & m_2 <j\leq m_{1}-1 .
\end{cases}
\end{equation}
It follows that $\dim \ker (A^j)= \dim \ker (A_{r+1}^j) $ for $1\leq j\leq m_1-1$, therefore, $A$ is conjugate to $A_{r+1 }$. 

By combining all cases, we conclude that there exists   $0\leq \ell\leq r+1-\delta_{s,0}$ such that $\g _A$ is isomorphic to $\g_{A_\ell}$. Furthermore, equations  \eqref{eq:caso(1)}, \eqref{eq:caso(2)} and \eqref{eq:caso(3)} imply that $\ell$ is unique. 
Note that $\g_{A_1}$ is the only $(m_1+1)$-step nilpotent Lie algebra, while all the remaining ones are $m_1$-step nilpotent. 

To complete the proof, we observe that the uniqueness of the integer $\ell$ above implies that $\g_{A_\ell}$ is not isomorphic to $\g_{A_{\ell'}}$ for $\ell\neq \ell'$. Finally, due to  Lemma \ref{lem:no-iso}, it follows that 
 $\g_N$ is not isomorphic to~$\g_{A_\ell}$.
\end{proof}

\smallskip

For each $n\geq 2$, let $\Sigma _{n-1}$ be as in Corollary \ref{cor:tuples} and fix a tuple in $\Sigma _{n-1}$ of the following form:  
\[ (r,m_1,\ldots, m_r, p_1,\ldots, p_r, s), \quad \;\; r>0,\, m_1 > \dots >m_r\geq 2, \;s\geq 0, \; p_k>0 , \] so that  \eqref{eq:dimn} is satisfied. 
According to Theorem \ref{thm:classif_nilp}$\rii$, the tuple above  gives rise to $r+2-\delta_{s,0}$  matrices in $\g\l(4n-1 , \R) $ of the form \eqref{eq:matrixA}, which correspond  to  different isomorphism classes of $4n$-dimensional nilpotent hypercomplex almost abelian Lie algebras. We denote by $\tilde{\Sigma}_{n-1}\subset  \g\l(4n-1 , \R) $ the set of nilpotent   matrices arising from all possible tuples in $\Sigma_{n-1}$. 
It follows from Lemma \ref{lem:no-iso} and  Corollary \ref{cor:tuples} that Lie algebras corresponding to  matrices arising from  different tuples are not isomorphic.  Note that if $\g_A$ is a  $4n$-dimensional hypercomplex nilpotent almost abelian  Lie algebra such that $\dim (\ker A)=4(n-1)$, then the corresponding matrix $B$ is equal to $0$ and, according to Theorem \ref{thm:classif_nilp}$\ri$, 
$\g_A$  is isomorphic to $\g_N$ with $N$ as in \eqref{eq:matrixN} for $s=n-1$. In other words, there is a unique, up to isomorphism, $4n$-dimensional hypercomplex nilpotent almost abelian  Lie algebra $\g_A$ satisfying $\dim (\ker A)=4(n-1)$. These observations are summarized in the next corollary.

\begin{corollary}\label{cor:param-nilp_hcx} 
The isomorphism classes of $4n$-dimensional nilpotent hypercomplex almost abelian Lie algebras $\g_A=\R e_0\ltimes _A \R^{4n-1}$ such that $\dim (\ker A) <4(n-1)$ are parametrized by $\tilde{\Sigma}_{n-1}$. If $\dim (\ker A) = 4(n-1)$ then $\g_A$ is isomorphic to  $\g_N$ with $N$ as in \eqref{eq:matrixN} for $s=n-1$. 
\end{corollary}

\begin{remark}\label{rem:gN} We observe that the 2-step nilpotent Lie algebra $\g_N$ in Theorem \ref{thm:classif_nilp} is isomorphic to $\g_3\times \R^{4n-8}$, where $\g_3$ is the 8-dimensional nilpotent Lie algebra from \cite[Theorem 5.1]{AB}.
\end{remark}

\begin{remark}\label{rem:latt-nilp}
It follows from Corollary \ref{cor:param-nilp_hcx} and Theorem \ref{thm:Mal} that all simply connected nilpotent hypercomplex almost abelian Lie groups admit lattices, since the matrices appearing in \eqref{eq:Al} have integer coefficients.
\end{remark}

Theorem \ref{thm:classif_nilp} and Corollary \ref{cor:param-nilp_hcx}  have their counterpart for  almost abelian Lie algebras with a complex structure, with  analogous proofs  (see Theorem \ref{thm:classif_nilp-complex} and Corollary \ref{cor:param-nilp-complex} in Appendix A for details).


We present below several consequences of Theorem \ref{thm:classif_nilp}.

\begin{corollary} There are $n-1$ isomorphism classes of $4n$-dimensional 2-step nilpotent  hypercomplex almost abelian Lie algebras.     
\end{corollary} 
\begin{proof}
Let $\g_A$ be a $4n$-dimensional 2-step nilpotent  hypercomplex almost abelian Lie algebra.  If $B=0$, then $r=0$ and $s=n-1$, hence $\g_A$ is isomorphic to $\g_N$ with $N$ as in \eqref{eq:matrixN}.   

If $B\neq 0$, we will show that there are $n-2$ isomorphism classes of Lie algebras, and this will complete the proof of the corollary. Since $\g_A$ is 2-step nilpotent we must have  $r=1, \; m_1=2$, and equation \eqref{eq:dimn} implies that $ s= n-2p_1-1$ with $p_1\geq 1$. It follows that $2p_1\leq n-1$ since $s\geq 0$. There are two possibilities:
\begin{itemize}
    \item $n$ is odd, therefore, $1\leq p_1\leq \frac{n-1}{2}$. 
    For each $1\leq p_1\leq \frac{n-1}{2}-1$  it follows that $s>0$ and Theorem \ref{thm:classif_nilp} implies that there are two isomorphism classes of $2$-step nilpotent Lie algebras. On the other hand, for $p_1=\frac{n-1}{2}$ we have that $s=0$ and there is just one isomorphism class. Therefore, there are $n-2$ isomorphism classes when $n$ is odd.
    \item $n$ is even, therefore, $1\leq p_1\leq \frac{n-2}{2}$ and $s=n-2p_1-1 >0 $. In this case,   there are two isomorphism classes of 2-step nilpotent Lie algebras for each value of $p_1$, that is, there are also $n-2$ isomorphism classes when $n$ is even. \qedhere
\end{itemize}    
\end{proof}  


In the next examples we apply Corollary \ref{cor:param-nilp_hcx} in dimensions 12 and 16.


\begin{example} In this example we apply Theorem \ref{thm:classif_nilp} to obtain the 12-dimensional nilpotent hypercomplex   almost abelian Lie algebras $\g_A$.  They will appear in Theorems  \ref{thm: classification1} and \ref{thm: classification2} below as $\s_{10}^0, \; \s_{16}^0$ and $\s_{16}^1$.

If $B=0$, then $\g_A$ is the 2-step nilpotent Lie algebra $\g_3\times \R^4$ (see Remark \ref{rem:gN}), which is denoted by  $\s_{10}^0$ in Theorem \ref{thm: classification1}. 

If $B\neq 0$, it follows from equation \eqref{eq:dimn} that $s=0, \; r=1, \, m_1=2, \; p_1=1$. Then, according to Theorem \ref{thm:classif_nilp}, $\g_A$ is isomorphic to $\g_{A_0}$  or   $\g_{A_1}$, and these are denoted by $\s_{16}^0$ and $\s_{16}^1$ in  Theorem \ref{thm: classification2}, respectively. 
\end{example}


\begin{example}
 By applying Theorem \ref{thm:classif_nilp} we show that  there are 6 isomorphism classes of   16-dimensional nilpotent  hypercomplex   almost abelian Lie algebras $\g_A$.  Indeed, if $B=0$, then $\g_A$ is isomorphic to  $\s_{10}^0 \times \R^4\cong \g_3 \times \R^8$. If $B\neq 0$, we have two possibilities for $\Sigma(B)$:
 \begin{itemize}
     \item $r=1, \; m_1=3, \; p_1=1, \; s=0$, which gives rise to two isomorphism classes of Lie algebras, one of them is 3-step nilpotent and the other one is 4-step nilpotent.
     \item $r=1, \; m_1=2, \; p_1=1, \; s=1$, which give rise to three  isomorphism classes of Lie algebras, two  of them are 2-step   nilpotent and the other one is 3-step nilpotent.
 \end{itemize}
\end{example}


We end this section by applying Theorem \ref{thm:classif_nilp} to obtain  necessary and sufficient conditions on a nilpotent matrix $M \in \mathfrak{gl}(4n-1,\R) $ so that the $4n$-dimensional almost abelian Lie algebra $\g=\R e_0 \ltimes _M \R^{4n-1}$ admits a hypercomplex structure.

Given a non-zero nilpotent matrix $M \in \mathfrak{gl}(4n-1,\R) $, consider its Jordan normal form:
\begin{equation}  \label{eq:matM}  
 M= j_{n_1}^{\oplus q_1}\oplus \dots \oplus j_{n_k}^{\oplus q_k} \oplus 0_d, \quad n_1> \dots > n_k\geq 2, \; k\geq 1,\; q_i >0 \; \text{for all } i, \; d\geq 0.   \end{equation}
where $j_{n_i}$ are  elementary  Jordan blocks  as in \eqref{eq:matrixjm}. 
Theorem \ref{thm:classif-jordan} below gives the necessary and sufficient conditions we are looking for, in terms of the non-negative integers $ n_i, \, q_i$ and $d$.

The next  lemma will be used in the proof of Theorem \ref{thm:classif-jordan} below. 
\begin{lemma}\label{lem:aux} Let $N$ (for $s>0$) and $N_\ell$ be the nilpotent matrices defined in \eqref{eq:matrixN}.  Then 
\begin{enumerate}
    \item[$\ri$]  $N$   is conjugate to $j_2^{\oplus 3}\oplus 0_{4s-3}$;
    \item[$\rii$] 
     $N_\ell$  is conjugate to $j_{m_\ell +1}^{\oplus 3}\oplus j_{m_\ell }^{\oplus (4p_\ell-3)}$.
\end{enumerate}

\end{lemma}
\begin{proof}
The proof of $\ri$ is straightforward.

 In order to prove $\rii$,  recall first that $\J _{m}$ is conjugate to $j_m^{\oplus 4}$. Then $\rii$ follows from the fact  that  
    \[  \left[ \begin{array}{c|ccc}
    0_3&  & & \\
    \hline 
    \begin{array}{c}
    U  \\ 
    0_{k\times 3}
    \end{array}& & \J_{m} & 
\end{array}\right], \qquad k=4(m-1),  \] 
 is conjugate to 
$\; j_{m +1}^{\oplus 3}\oplus j_{m}$.
\end{proof}


\begin{theorem}\label{thm:classif-jordan} Let $M\in  \mathfrak{gl}(4n-1,\R)$ be a non-zero nilpotent matrix as in \eqref{eq:matM}. Then the $4n$-dimensional nilpotent almost abelian Lie algebra  $\g=\R e_0 \ltimes _M \R^{4n-1}$  admits a hypercomplex structure if and only if any of the  following conditions is satisfied:
\begin{enumerate}
\item[$\ri$] $n_k=2, \; q_k \equiv 3 \pmod{4}$, $  d \equiv 1 \pmod{4}$ and $q_i\equiv 0 \pmod{4}$ for  $i\neq k$;
    \item[$\rii$] $q_i\equiv 0 \pmod{4}$ for all $i$ and $d\equiv 3 \pmod{4}$;
    \item[$\riii$] there exists $2\leq t \leq k$ such that $n_{t-1}= n_{t} +1$, $q_{t-1}\equiv 3 \pmod{4}, \; q_{t } \equiv 1 \pmod{4},\; q_i\equiv 0 \pmod{4}$ for  $i\notin \{t-1, t\}$ and $d\equiv 0 \pmod{4}$.
    \end{enumerate}
\end{theorem} 

\begin{proof}
Assume first that $\g$ admits a hypercomplex structure. Then, it follows from Remark \ref{rem:nilp-sim} and Theorem \ref{thm:classif_nilp} that $M$ must be conjugate to one of the following matrices:
\begin{enumerate}
    \item[(1)] $N$ as in \eqref{eq:matrixN} for some $s>0$,
    \item [(2)] $A_0$ as in \eqref{eq:Al},
    \item[(3)] $A_l$ as in \eqref{eq:Al}, for $1\leq l\leq r+1-\delta_{s,0}$. 
\end{enumerate} 
We determine next the Jordan normal form of each of these matrices. 

\textsl{ Case (1):} Assume that $M$ is conjugate to $N$ for some $s>0$. It follows from Lemma \ref{lem:aux}$\ri$ that $N$  is conjugate to $j_2^{\oplus 3}\oplus 0_{4s-3}$, so we have $k=1$, $n_1=2$, $q_1=3$ and $d=4s-3$. Therefore, $M$ satisfies $\ri$.

\textsl{Case (2):} The matrix  $A_0$ is conjugate to \[j_{m_1}^{\oplus 4p_1}\oplus \dots \oplus j_{m_r}^{\oplus 4p_r} \oplus 0_{4s+3},\] so if $M$ and $A_0$ are conjugate we must have  $k=r$, $q_i=4 p_i$,  $d= 4s+3$ and  $\rii$ holds.

\textsl{Case (3):} 
If $l=1$, it follows from Lemma \ref{lem:aux}$\rii$ that $A_1$ is conjugate to \[ j_{m_1 +1}^{\oplus 3}\oplus j_{m_1 }^{\oplus (4p_1-3)}\oplus j_{m_2 }^{\oplus 4p_2}\oplus \dots \oplus j_{m_r }^{\oplus 4p_r}\oplus 0_{4s}.\] Therefore, if $M$ is conjugate to $A_1$ we have 
$k=r+1$, $n_1=m_1+1$, $n_2=m_1=n_1-1$, $n_i= m_{i-1}$ for $3\leq r\leq k+1$, $q_1=3$, $q_2=4p_1-3 $, $q_i=4 p_{i-1}$ and $d=4s$. Hence, $\riii$ is satisfied for $t=2$. 

Assume next that $M$  is conjugate to $A_l$ for some $2\leq l\leq r$. We apply again Lemma \ref{lem:aux}$\rii$ to obtain that $A_l$  is conjugate to
  \[ j_{m_1 }^{\oplus 4 p_1}\oplus \dots \oplus j_{m_{l-1} }^{\oplus 4 p_{l-1}} \oplus 
  j_{m_{l}+1 }^{\oplus 3} \oplus j_{m_{l} }^{\oplus (4p_l-3)}\oplus j_{m_{l+1} }^{\oplus 4p_{l+1}}\oplus \dots \oplus   j_{m_r }^{\oplus 4p_r}\oplus 0_{4s}. \]  
  There are two possibilities: either $m_l+1=m_{l-1}$ or $m_l+1<m_{l-1}$. In the first case, $A_l$ is conjugate to 
  \[  j_{m_1 }^{\oplus 4 p_1}\oplus \dots \oplus j_{m_{l-2} }^{\oplus 4 p_{l-2}} \oplus   
  j_{m_l +1 }^{\oplus (4 p_{l-1}+3)} \oplus 
   j_{m_{l} }^{\oplus (4p_l-3)}\oplus j_{m_{l+1} }^{\oplus 4p_{l+1}}\oplus \dots \oplus   j_{m_r }^{\oplus 4p_r}\oplus 0_{4s},
  \] 
  and we have that $k=r$, $n_i=m_i$ and $d=4s$. It follows that $\riii$ is satisfied for $t=l$. On the other hand, if $m_l+1<m_{l-1}$, $A_l $ is conjugate to 
  \[  j_{m_1 }^{\oplus 4 p_1}\oplus \dots \oplus j_{m_{l-1} }^{\oplus 4 p_{l-1}} \oplus   
  j_{m_{l}+1 }^{\oplus 3} \oplus 
   j_{m_{l} }^{\oplus (4p_l-3)}\oplus j_{m_{l+1} }^{\oplus 4p_{l+1}}\oplus \dots \oplus   j_{m_r }^{\oplus 4p_r}\oplus 0_{4s}.
  \] It follows that $k=r+1$, $n_i=m_i$ for $1\leq i\leq l-1$, $n_l=m_l+1$, $n_i=m_{i-1}$ for $l+1\leq i\leq r$, and $\riii$ holds for $t= l+1$. 
  
  Finally, assume that $M$ is conjugate to $A_{r+1}$ (so that $s>0$).  We observe that $A_{r+1}$ is conjugate to \[
  j_{m_1}^{\oplus 4p_1}\oplus \dots \oplus j_{m_r}^{\oplus 4p_r}\oplus j_2^{\oplus 3} \oplus 0_{4s-3},
  \] 
  therefore,  $M$ satisfies $\ri$.
  
  For the converse, we show that if any   of the conditions in the statement is satisfied, then $M$ is conjugate to one of the matrices in cases (1), (2) or (3) above, therefore, it will follow from Remark \ref{rem:nilp-sim} and Theorem \ref{thm:classif_nilp} that  
  $\g$ admits a hypercomplex structure.  
  
  Assume first that $\ri$ holds and let $p_i, \, n_i, \, 1\leq i \leq k, \, s>0$, such that $q_i=p_i$ for $1\leq i \leq k$, $q_k=4p_k +3$ and $d=4s-3$. 
  If $p_k=0$ then $M=j_{n_1}^{\oplus 4 p_1} \oplus \dots \oplus j_{n_{k-1}}^{\oplus 4 p_{k-1}}\oplus  j_2^{\oplus 3}\oplus 0_{4s-3}$. Since  $j_2^{\oplus 3}\oplus 0_{4s-3}$ is conjugate to $N$ (see Lemma \ref{lem:aux}$\ri$), then  $M$ is conjugate to $j_{n_1}^{\oplus 4 p_1} \oplus \dots \oplus j_{n_{k-1}}^{\oplus 4 p_{k-1}}\oplus N$, which, in turn, is conjugate to $A_{r+1}$ from \eqref{eq:Al} with $r=k-1$. If $p_k>0$ then it follows that $M$ is conjugate to $j_{n_1}^{\oplus 4 p_1} \oplus \dots \oplus j_{n_{k-1}}^{\oplus 4 p_{k-1}}\oplus j_2 ^{\oplus 4p_k} \oplus N$. Therefore, $M$ is conjugate to $A_{r+1}$ for $r=k$.
  
  In case $\rii$ is satisfied, then there exist $p_i, \, 1\leq i\leq k, \; s\geq 0$ such that $q_i=4p_i$ and $d=4s+3$. It is straightforward that $M$ is conjugate to $A_0$ from \eqref{eq:Al} with $r=k$. 
  
  Finally, if $\riii$ holds, let $2\leq t\leq k$ as in the statement. There  exist $p_i, \; 1\leq i\leq k, \, p_t>0$, and $s\geq 0$ such that $q_i=4p_i$ for $i\notin \{ t-1 ,t\}$, $q_{t-1}=4p_{t-1}+3$, $q_t=4p_t -3$ and $d=4s$. Therefore, 
  \begin{eqnarray*}
       M&= &  j_{n_1 }^{\oplus 4 p_1}\oplus \dots \oplus j_{n_{t-2} }^{\oplus 4 p_{t-2}} \oplus   
  j_{n_{t}+1 }^{\oplus (4p_{t-1}+3)} \oplus 
   j_{n_{t} }^{\oplus (4p_t-3)}\oplus j_{n_{t+1} }^{\oplus 4p_{t+1}}\oplus \dots \oplus   j_{n_k }^{\oplus 4p_k}\oplus 0_{4s}, \\
   & =& j_{n_1 }^{\oplus 4 p_1}\oplus \dots \oplus j_{n_{t-2} }^{\oplus 4 p_{t-2}} \oplus   j_{n_{t}+1 }^{\oplus 4p_{t-1}} \oplus \left(
  j_{n_{t}+1 }^{\oplus 3} \oplus 
   j_{n_{t} }^{\oplus (4p_t-3)}\right) \oplus j_{n_{t+1} }^{\oplus 4p_{t+1}}\oplus \dots \oplus   j_{n_k }^{\oplus 4p_k}\oplus 0_{4s}. 
  \end{eqnarray*}
   From Lemma \ref{lem:aux}(ii) we have that  $j_{n_{t}+1 }^{\oplus 3} \oplus j_{n_{t} }^{\oplus (4p_t-3)} $ is conjugate to $N_t$ as in  \eqref{eq:matrixN}. Therefore, if $n_t+1 \neq n_{t-2}$ then $M$ is conjugate to $ A_t$ with $r=k$, $n_i=4m_i, \; 1\leq i\leq k$. If $n_t= n_{t-2}$, then $M$ is conjugate to  
   \[  M=j_{n_1 }^{\oplus 4 p_1}\oplus \dots \oplus j_{n_{t-2} }^{\oplus 4( p_{t-2}+p_{t-1})} \oplus    N_t \oplus j_{n_{t+1} }^{\oplus 4p_{t+1}}\oplus \dots \oplus   j_{n_k }^{\oplus 4p_k}\oplus 0_{4s},  \]
   which is conjugate to a matrix analogous to $A_t$ with $r=k-1$. This concludes the proof.
\end{proof}


\section{Classification in dimension 12}\label{sect:dim12}

In this section we provide the classification of the 12-dimensional almost abelian Lie algebras $\s$ that admit a hypercomplex structure. In order to perform this classification, we will make use of the quaternionic Jordan form described in \S \ref{sec:quat_linear}.

First, we set some notation. Given a basis $f_0, \dots , f_{11}$ of $\s$, let  $f^0, \dots , f^{11}$ be the dual basis of $\s^*$. Since $\s$ is almost abelian, the Lie bracket on $\s$ is determined by  $\al_0, \dots , \al _{11}\in \s^*$ such that $df^i=-f^0\wedge \al_i$. In the next theorems, we will denote this Lie algebra by 
$-f^0\wedge( \al_0, \dots , \al _{3} \, |\, \al_4, \dots , \al _{7} \, |\, \al_8, \dots , \al _{11} )$.


We begin now with the classification. Let $\s=\R e_0\ltimes_A \R^{11}$ be a $12$--dimensional almost abelian Lie algebra equipped with a hypercomplex structure. It follows from Theorem \ref{thm:characterization} that 
the matrix $A$ in \eqref{eq:matrixA} takes the  form
\begin{equation}\label{eq:matrix-dim8} 
A=\left[
	\begin{array}{ccc|ccc}       
		\mu & & & & &\\
		& \mu & & & 0 & \\
		& & \mu & & & \\
		\hline
		| & | & |& & & \\
		v_1 & v_2 & v_3 & & B &\\
		| & | & | & & &
	\end{array}
	\right], \quad \text{with} \quad v_\alpha\in\R^8,\, v_\al= J_\al v_0, \,  B\in \mathfrak{gl} (2,\H).
\end{equation} 
The idea of the classification is to obtain a simpler form of $B$ by  analyzing the different possibilities for the quaternionic Jordan form of $\sigma(B)$ (see \eqref{eq:BH}). Since $\sigma(B)\in M_2(\H)$, its possible Jordan forms are:
\[ \begin{bmatrix} 
\lambda_1 & 0\\
0&\lambda _2
\end{bmatrix} \qquad \text{ or } \qquad
 \begin{bmatrix} 
\lambda_1 & 0\\
1&\lambda_1 
\end{bmatrix},
\]
where $\lambda_1=a+ib, \lambda_2=c+id, $ with $a,b,c,d \in \R,\, b,d\geq 0$. Applying $\sigma^{-1}$ we obtain that in an ordered basis $\{ u_0,\ldots, u_3, w_0, \ldots ,  w_3\}$  with $u_\al =J_\al u_0, \,  w_\al =J_\al w_0$, the matrix $B$ can be written as:
\begin{equation}\label{eq:casos} B_1= \left[\begin{array}{cccc|cccc} a& -b& && &&& \\
b&a&&&&&& \\
&&a&b&&&& \\
&&-b&a&&&&\\
\hline
&&&& c&-d&& \\
&&&&d&c&& \\
&&&& &&c&d \\
&&&& &&-d&c
\end{array}\right] \;\text{ or } \;
B_2= \left[\begin{array}{cccc|cccc} a& -b& && &&& \\
b&a&&&&&& \\
&&a&b&&&& \\
&&-b&a&&&&\\
\hline
1&&&& a&-b&& \\
&1&&&b&a&& \\
&&1&& &&a&b \\
&&&1& &&-b&a
\end{array}\right].
\end{equation}
It follows from Lemma \ref{lem:no-iso} that Lie algebras arising from $B_1$ cannot be isomorphic to those arising from $B_2$.

We next analyze each case separately. We set the following notation: \begin{equation} \label{eq:V1-V2} V_1:=\text{span}\{u_0,\ldots,u_3\}, \qquad  V_2:=\text{span}\{w_0,\ldots,w_3\}.\end{equation} 
Note that $V_1$ and $V_2$ are two $\hcx$-invariant subspaces of $\h$ such that $\h=V_1\oplus V_2$. 

\smallskip

We begin the classification in dimension 12 with the case $B=B_1$. 

\begin{theorem} \label{thm: classification1}
Let $\s=\R e_0\ltimes_A \R^{11}$ be a $12$-dimensional almost abelian Lie algebra admitting a hypercomplex structure $\hcx$, with $A$ as in \eqref{eq:matrixA} and $B=B_1$ from \eqref{eq:casos}. Then $\s$ is isomorphic to one and only one of the following Lie algebras:
\[ \begin{array}{rll}
 \s_1^{a,c,d}: &   -f^0\wedge &  \!\!\!\!(0,0,0,0\, |\, af^4-f^5,f^4+af^5, af^6+f^7,-f^6+af^7\, | \\ & &cf^8-df^9,df^8+cf^9, cf^{10}+df^{11},-df^{10}+cf^{11}),  \quad a,c \in \R, \, d>1,\\
\s_2^{a,c}: &   -f^0\wedge &  \!\!\!\!(0,0,0,0\, |\, af^4-f^5,f^4+af^5, af^6+f^7,-f^6+af^7\, | \\ & &cf^8-f^9,f^8+cf^9, cf^{10}+f^{11},-f^{10}+cf^{11}),  \quad  \text{\scriptsize{ $(0\leq a\leq c)$ or $(a<0<c,\, |a|\leq |c|)$}},\\
\s_3^{a,b,c,d}: &   -f^0\wedge &  \!\!\!\!(0,f^1,f^2,f^3\, |\, af^4-bf^5,bf^4+af^5, af^6+bf^7,-bf^6+af^7\, | \\ & &cf^8-df^9,df^8+cf^9, cf^{10}+df^{11},-df^{10}+cf^{11}),  \quad a,c \in \R, \, 0<b<d,\\
\s_4^{a,b,c}: &   -f^0\wedge &  \!\!\!\!(0,f^1,f^2,f^3\, |\, af^4-bf^5,bf^4+af^5, af^6+bf^7,-bf^6+af^7\, | \\ & &cf^8-bf^9,bf^8+cf^9, cf^{10}+bf^{11},-bf^{10}+cf^{11}),  \quad a\leq c, \, b>0,\\
 \s_5^{a,c}: &   -f^0\wedge &  \!\!\!\!(0,0,0,0\, |\, af^4,af^5, af^6,af^7\, | \\ & &cf^8-f^9,f^8+cf^9, cf^{10}+f^{11},-f^{10}+cf^{11}),  \quad \text{\scriptsize{ $(a\neq 0,\, c \in \R)$ or $(a=0, \, c\geq 0)$}},\\
 \s_6^{c}: &   -f^0\wedge &  \!\!\!\!(0,0,0,0\, |\,0,f^1,f^2,f^3\, | \\ & &cf^8-f^9,f^8+cf^9, cf^{10}+f^{11},-f^{10}+cf^{11}),  \quad c \geq 0,\\
 \s_7^{a,c,d}: &  -f^0\wedge &  \!\!\!\!(0,f^1,f^2,f^3\, |\, af^4,af^5,a f^6,af^7\, | \\ & &cf^8-df^9,df^8+cf^9, cf^{10}+df^{11},-df^{10}+cf^{11}),  \quad 
 a,c \in \R, \, d>0,\\
 \s_8^{c,d}: &   -f^0\wedge &  \!\!\!\!(0,f^1,f^2,f^3\, |\, f^4,f^1+f^5,f^2+ f^6,f^3+f^7\, | \\ & &cf^8-df^9,df^8+cf^9, cf^{10}+df^{11},-df^{10}+cf^{11}),  \quad c \in \R, \, d>0,\\
 \s_9^{c}: &  -f^0\wedge &  \!\!\!\!(0,0,0,0\, |\, f^4,f^5, f^6,f^7\, | \,  cf^8,cf^9, cf^{10},cf^{11}),  \quad 
 |c|\leq 1,\\
 \s_{10}^{c}: &  -f^0\wedge &  \!\!\!\!(0,0,0,0\, |\, 0,f^1, f^2,f^3\, | \,  cf^8,cf^9, cf^{10},cf^{11}),  \quad 
 c=0 \text{ or } 1,\\
 \s_{11}^{a,c}: &  -f^0\wedge &  \!\!\!\!(0,f^1,f^2,f^3\, |\, af^4,af^5, af^6,af^7\, | \,  cf^8,cf^9, cf^{10},cf^{11}),  \quad a\leq c,\\
 \s_{12}^{c}: &   -f^0\wedge &  \!\!\!\!(0,f^1,f^2,f^3\, |\, f^4,f^1+f^5,f^2+ f^6,f^3+f^7\, | \,  cf^8,cf^9, cf^{10},cf^{11}),  \quad c \in \R.
 \end{array}\]
\end{theorem}
\begin{proof}
We have $B=B_1$ in \eqref{eq:casos}, and recall that $b, d\geq 0$. By permuting the elements in the basis if necessary, we may suppose that $b\leq d$. We analyze several cases.
\begin{enumerate}
    \item[(i)] Assume that $b>0,\, d>0$. According to Proposition \ref{prop:v=0}, since the eigenvalues of $B$ are not real, we may assume $v_0 =0$ (with $v_0$ as in Theorem \ref{thm:characterization}). We will use repeatedly Lemma \ref{lem:ad-conjugated} when we divide the matrix $A$ by a non-zero scalar.
    \begin{itemize}
        \item $\mu=0$: We may assume $b=1$. If $d>1$, we obtain  pairwise non-isomorphic Lie algebras denoted by $\s_1^{a,c,d}$. If $d=1$, we will denote these Lie algebras by $\s_2^{a,c}$. If $a\leq 0,\, c\leq 0$, multiplying $A$ by $(-1)$ and permuting the blocks if necessary, we may assume $0\leq a\leq c$. The remaining case is when one of the parameters is negative and the other one is positive. Multiplying $A$ by $(-1)$ and permuting the blocks if necessary, we may assume $a<0<c$ and $|a|\leq |c|$. All the Lie algebras $\s_2^{a,c}$, with these restrictions on the parameters, are pairwise non-isomorphic. 
        \item $\mu\neq 0$: We may assume $\mu =1$. If $b<d$, all these Lie algebras, denoted by $\s_3^{a,b,c,d}$, are pairwise non-isomorphic. On the other hand, if $b=d$, we may take $a\leq c$  obtaining in this way pairwise  non-isomorphic Lie algebras denoted by $\s_4^{a,b,c}$. 
\end{itemize}

\item[(ii)] Assume that $b=0,\, d>0$.
\begin{itemize}    
\item $\mu=0$: We may assume $d=1$. 
If $v_0=0$, we denote these Lie algebras by $\s_5^{a,c}$. For any choice of $a\neq 0$ and $c\in\R$ we obtain non-isomorphic Lie algebras, but, on the other hand, when $a=0$ multiplying $A$ by $(-1)$ gives rise to an isomorphic Lie algebra, so that we may assume $c\geq 0$ and the Lie algebras $\s_5^{0,c}$ are pairwise non-isomorphic for different values of $c\geq 0$. If $v_0\neq 0$, according to Proposition \ref{prop:v=0} we may take $a=0$. Moreover, we can suppose that $v_0\in V_1$ (see \eqref{eq:V1-V2}), due to Lemma \ref{lem:vnot0}. Since $V_1$ is $\hcx$-invariant, we may choose $\{v_0, \dots , v_3\}$ as a basis of $V_1$ and therefore, the matrix $A$ can be written as
\[ A=\left[\begin{array}{c|c}
          0_3   &  \\
          \hline 
        U   &  0_4
        \end{array}\right] \oplus \left[ \begin{array}{cccc}
 c&-1&& \\
             1&c&& \\
 &&c&1\\
 &&-1&c
\end{array}\right], \quad\text{ with } U \text{ as in  } \eqref{eq:matrixU}.\] 
Multiplying $ A$ by $(-1)$ and permuting the blocks of the second matrix if necessary, we may assume that $c\geq 0$ since this does not change the conjugacy class of the first matrix. Therefore, these Lie algebras are pairwise non-isomorphic for different values of $c\geq 0$ and we denote them by $\s_6^{c}$.
                \item $\mu\neq 0$: We may assume $\mu=1$. If $v_0=0$,  for an arbitrary choice of $a, \, c$ and $d>0$ we obtain non-isomorphic Lie algebras denoted by $\s_7^{a,c,d}$. If $v_0\neq 0$, according to Proposition \ref{prop:v=0} and Lemma \ref{lem:vnot0} we may assume $a=1$ and $v_0\in V_1$. 
        Since $V_1$ is $\hcx$-invariant, we may choose $\{v_0, \dots , v_3\}$ as a basis of $V_1$ and therefore, the matrix $A$ can be written as
        \[ A=\left[\begin{array}{c|c}
          I_3   &  \\
          \hline 
        U   &  I_4
        \end{array}\right] \oplus \left[ \begin{array}{cccc}
 c&-d&& \\
             d&c&& \\
 &&c&d\\
 &&-d&c
\end{array}\right], \quad\text{ with } U \text{ as in  } \eqref{eq:matrixU}.\]
These Lie algebras, denoted by $\s_8^{c,d}$, are pairwise non-isomorphic for different values of $c$ and $d>0$.
\end{itemize}

\item[(iii)] Assume that $b=d=0$. 
\begin{itemize}
    \item $\mu=0$: Let us suppose first that $v_0=0$, then $a\neq 0$ or $c\neq 0$, otherwise, the Lie algebra would be abelian. Since $a$ and $c$ are interchangeable, we may assume $a=1$. These Lie algebras, denoted by $\s_9^{c}$, are pairwise non-isomorphic provided that $|c|\leq 1$. Next, if $v_0\neq 0$,  according to Proposition \ref{prop:v=0} we may assume that $a=0$ or $c=0$. Without loss of generality, we may take $a=0$ and due to  Lemma \ref{lem:vnot0}, $v_0\in V_1$.  Moreover, using the basis $\{v_0,\dots , v_3\}$ for $V_1$, we have that
    \[ A=\left[\begin{array}{c|c}
          0_3   &  \\
          \hline 
        U  &  0_4
        \end{array}\right] \oplus cI_4, \quad\text{ with } U \text{ as in  } \eqref{eq:matrixU}  \text{ and }  c\in \R.\]
 The corresponding Lie algebras are denoted by $\s_{10}^{c}$. Note that  $\s_{10}^{0}$   is $2$-step nilpotent. On the other hand, when  $c\neq 0$ we may divide by $c$ obtaining a matrix conjugate to $A$ with $c=1$ and therefore $\s_{10}^{c}$ is isomorphic to $\s_{10}^{1}$. 
\item $\mu\neq 0$: We may assume $\mu=1$. If $v_0=0$, for any choice of $a$ and $c$ with $a\leq c$ we obtain pairwise non-isomorphic Lie algebras, denoted by $\s_{11}^{a,c}$. If $v_0\neq 0$, then according to Proposition \ref{prop:v=0} 
we may assume that $a=1$ or $c=1$. Without loss of generality, we may take $a=1$ and due to  Lemma \ref{lem:vnot0}, $v_0\in V_1$.  Moreover, using the basis $\{v_0,\dots , v_3\}$ for $V_1$, 
    \[ A=\left[\begin{array}{c|c}
          I_3   &  \\
          \hline 
        U   &  I_4
        \end{array}\right] \oplus cI_4, \quad\text{ with } U \text{ as in  } \eqref{eq:matrixU}  \text{ and }  c\in \R.\]
These Lie algebras, denoted by $\s_{12}^{c}$, are pairwise non-isomorphic. 
\end{itemize}
\end{enumerate}
Lie algebras from two different families are not isomorphic as a consequence of Lemma \ref{lem:no-iso}. 
\end{proof}


Now we move to the second case, namely, $B=B_2$, 
and we start by proving the following lemma.

\begin{lemma}\label{lem:U2} Let $A$ be  as in \eqref{eq:matrix-dim8} with $B=B_2$ from  \eqref{eq:casos} and $\lambda_1=\mu$. Then
\begin{enumerate}
    \item[$\ri$] if    $v_0\in V_2$, $\g_A$ is isomorphic to $\g_{\tilde A}$, where $\tilde{A}=\left[ \begin{array}{c|c} \mu I_3 & \\
    \hline 
    & B_2    
    \end{array}\right]$, 

    \smallskip

    \item[$\rii$]  if    $v_0\notin V_2$, $\g_A$ is isomorphic to $\g_{\tilde A}$, where $\tilde{A}=\left[ \begin{array}{c|c} \mu I_3 & \\
    \hline 
    \begin{array}{c} U \\ 
     0_{4\times 3} \end{array}
    & B_2    
    \end{array}\right]$, with $U$ as in~\eqref{eq:matrixU}.
\end{enumerate}
\end{lemma}

\begin{proof}
If  $v_0\in V_2=\operatorname{Im} (B_2-\mu I)$,  $\ri$ follows from  Proposition \ref{prop:v=0}.\\
Assume next that  $v_0\notin V_2$. Lemma \ref{lem:vnot0} implies that we may assume $v_0\in V_1$, so that $A$ takes the form   $A=\left[ \begin{array}{c|c} \mu I_3
& \\
    \hline 
    \begin{array}{c} V \\ 
     0_{4\times 3} \end{array}
    & B_2    
    \end{array}\right]$, where $V$ is the $4\times 3$ matrix whose columns are $v_\al, \, \al=1,2,3$, which are linearly independent. Since $\dim \ker (A-\mu I)=\dim \ker (\tilde A-\mu I)$ and $(A-\mu I)^2=(\tilde A-\mu I)^2=0$, it follows that $A$ is conjugate to $\tilde A$ as in $\rii$, hence $\g_A$ is isomorphic to $\g_{\tilde A}$.  
\end{proof}

\smallskip

\begin{theorem} \label{thm: classification2}
Let $\s=\R e_0\ltimes_A \R^{11}$ be a $12$-dimensional almost abelian Lie algebra admitting a hypercomplex structure $\hcx$, with $A$ as in \eqref{eq:matrixA} and $B=B_2$ from \eqref{eq:casos}. Then $\s$ is isomorphic to one and only one of the following Lie algebras:
\[ \begin{array}{rll}
 \s_{13}^{a}: &   -f^0\wedge &  \!\!\!\!(0,0,0,0\, |\, af^4-f^5,f^4+af^5, af^6+f^7,-f^6+af^7\, | \\ & &f^4+af^8-f^9,f^5+f^8+af^9, f^6+af^{10}+f^{11},f^7-f^{10}+af^{11}),  \quad a \geq 0, \\
 \s_{14}^{a,b}: &   -f^0\wedge &  \!\!\!\!(0,f^1,f^2,f^3\, |\, af^4-bf^5,bf^4+af^5, af^6+bf^7,-bf^6+af^7\, | \\ & &f^4+af^8-bf^9,f^5+bf^8+af^9, f^6+af^{10}+bf^{11},f^7-bf^{10}+af^{11}),   \text{\scriptsize{ $a \in \R,\, b>0$}},\\
 \s_{15}: &   -f^0\wedge &  \!\!\!\!(0,0,0,0\, |\, f^4,f^5, f^6,f^7\, | \, f^4+f^8,f^5+f^9, f^6+f^{10},f^7+f^{11}), \\
\s_{16}^s: &   -f^0\wedge &  \!\!\!\!(0,0,0,0\, |\, 0,sf^1,sf^2,sf^3\, | \, f^4,f^5, f^6,f^7), \quad s =0 \text{ or } 1,\\
\s_{17}^a: &   -f^0\wedge &  \!\!\!\!(0,f^1,f^2,f^3\, |\, af^4,af^5, af^6,af^7\, | \, f^4+af^8,f^5+af^9, f^6+af^{10},f^7+af^{11}), \, \text{\scriptsize{$a\in \R$}},\\
 \s_{18}: &   -f^0\wedge &  \!\!\!\!(0,f^1,f^2,f^3\, |\, f^4,f^1+f^5,f^2+ f^6,f^3+f^7\, | \, f^4+f^8,f^5+f^9, f^6+f^{10},f^7+f^{11}).
\end{array}\]
\end{theorem}

\begin{proof} 

We have $B=B_2$ in \eqref{eq:casos}. Recall that $b\geq 0$. We analyze several cases. 

\begin{enumerate}
\item[(i)] Assume $b>0$, then $B-\mu I$ is invertible and it follows from Proposition \ref{prop:v=0} that we may take $v_0=0$. 
\begin{itemize}
    \item $\mu=0$: We can assume $b=1$ and these Lie algebras are denoted by  $\s_{13}^{a}$. Multiplying $A$ by $(-1)$ and permuting blocks if necessary, we may assume that $a\geq 0$ and they are pairwise non-isomorphic for different values of $a\geq 0$. 
    \item $\mu\neq 0$: We may take $\mu=1$ and for any choice of $a$ and $b>0$ we obtain pairwise non-isomorphic Lie algebras, denoted by $\s_{14}^{a,b}$. 
\end{itemize}

\item[(ii)]  Assume $b=0$. 
\begin{itemize}
\item $\mu=0$: If $a\neq 0$, we may assume $a=1$ and $v_0=0$, so in this case there is only one Lie algebra up to isomorphism, denoted by $\s_{15}$. On the other hand, if $a=0$ we may have $v_0\in V_2$ or $v_0\notin V_2$. According to Lemma \ref{lem:U2}, $\g_A$ is isomorphic to $\g_{\tilde A}$ where 
    \[
    \tilde A=\left[\begin{array}{c|c|c}
          0_3   &  &\\
          \hline 
           &  0_4 & \\
        \hline
         & I_4&0_4
        \end{array}\right], \quad\text{ or }\quad 
        \tilde A=\left[\begin{array}{c|c|c}
          0_3   &  &\\
          \hline 
        U   &  0_4 & \\
        \hline
         & I_4&0_4
        \end{array}\right], \quad\text{ with } U \text{ as in  } \eqref{eq:matrixU}, 
    \]
giving rise to the Lie algebras  $\s_{16}^0$ and $\s_{16}^1$, respectively. The Lie algebra $\s_{16}^0$ is $2$-step nilpotent, whereas $\s_{16}^1$ is $3$-step nilpotent.

\item $\mu\neq 0$:  We may assume $\mu=1$. If $a\neq 1$, we can take $v_0=0$ and we obtain the Lie algebras denoted by $\s_{17}^a$, which are pairwise non-isomorphic for different values of $a$. On the other hand, if $a=1$, then either $v_0\in V_2$  or $v_0\notin V_2$, and according to Lemma \ref{lem:U2}, $\g_A$ is isomorphic to $\g_{\tilde A}$ where 
    \[ \tilde A=\left[\begin{array}{c|c|c}
          I_3   &  &\\
          \hline 
           &  I_4 & \\
        \hline
         & I_4&I_4
        \end{array}\right] \quad 
        \text{ or } \quad 
        \tilde A=\left[\begin{array}{c|c|c}
          I_3   &  &\\
          \hline 
        U   &  I_4 & \\
        \hline
         & I_4&I_4
        \end{array}\right], \quad \text{ with } U \text{ as in  } \eqref{eq:matrixU},
        \] 
giving rise to the Lie algebras $\s_{17}^1$ and $\s_{18}$, respectively.
\end{itemize}
\end{enumerate}
Lie algebras from two different families are not isomorphic as a consequence of Lemma \ref{lem:no-iso}.
\end{proof}

\smallskip

\begin{corollary}\label{cor:unimod}
Among the Lie algebras in Theorems \ref{thm: classification1} and \ref{thm: classification2}, we have that 
\begin{itemize}
    \item the unimodular ones are: $\s_1^{a,c,d}$ and $\s_2^{a,c}$ with $a+c=0$, $\s_3^{a,b,c,d}$ and  $\s_4^{a,b,c}$ with $a+c=-\frac34$, $\s_5^{a,c}$ with $a+c=0$, $\s_6^0$, $\s_7^{a,c,d}$ with $a+c=-\frac34$, $\s_8^{-\frac74,d}, \,\s_9^{-1},\, \s_{10}^0$, $\s_{11}^{a,c}$ with $a+c=-\frac34$, $\s_{12}^{-\frac74},\, \s_{13}^0, \, \s_{14}^{-\frac38,b}, \,  \s_{16}^s $ and $  \s_{17}^{-\frac38}$,

    \item the completely solvable ones are: $\s_9^{c},\, \s_{10}^c$, $\s_{11}^{a,c}, \, \s_{12}^{c},\, \s_{15}, \,   \s_{16}^s, \, \s_{17}^{a}$ and $\s_{18}$,

    \item the nilpotent ones are: $\s_{10}^0$ (2-step), $\s_{16}^0$ (2-step) and $\s_{16}^1$ (3-step).
\end{itemize}
\end{corollary}


\begin{remark}\label{rem:hkt}
 It was proved in \cite[Proposition 4.1]{AB} that a hyperhermitian metric on a hypercomplex almost abelian Lie algebra is HKT if and only if $B$ is skew-adjoint and $v_0=0$, where $B$ and $v_0$ are as in Theorem \ref{thm:characterization}. Moreover, the HKT metric is hyper-Kähler if and only if $\mu=0$. 
 
In the 12-dimensional case, we have that the Lie algebras admitting an HKT metric only appear in the case $B=B_1$ with $a=c=0$. Thus, they are $\s_3^{0,b,0,d}$, $\s_4^{0,b,0}$ and $\s_7^{0,0,d}$ (when the metric is not hyper-Kähler), and $\s_1^{0,0,d}$, $\s_2^{0,0}$ and $\s_5^{0,0}$ (when the metric is hyper-Kähler).
\end{remark}



\section{Lattices in 12-dimensional hypercomplex almost abelian Lie groups}\label{sec:lattices}

In this section we will determine which of the simply connected  Lie groups corresponding to the hypercomplex  Lie algebras listed in Theorems \ref{thm: classification1} and \ref{thm: classification2} have lattices. We begin by proving a useful result about polynomials with integer coefficients.

\begin{lemma}\label{lem:odd} 
Let $p\in\Z[x]$ be a monic polynomial with $|p(0)|=1$. Then 
\begin{enumerate}
    \item[$\ri$] if $p$ has exactly one root $\alpha$ of odd multiplicity then $\al=\pm 1$;
    \item[$\rii$] if $\alpha\in\C$ and $\bar{\alpha}\neq \alpha$ are the only roots of $p$  with odd multiplicity, then $|\alpha |=1$. 
\end{enumerate}
\end{lemma}

\begin{proof}
To prove $\ri$ we proceed by induction on the degree of $p$. We first observe that 
if $p$ has exactly one root $\alpha$ of odd multiplicity, then $\alpha\in \R$ 
since $p$ has real coefficients. 

If $\deg p=1$ then $p(x)=x-\al$. In particular, $\al\in\Z$ and $1=|p(0)|=|\al|$, therefore $\al=\pm 1$.

If $\deg p>1$ it follows that, indeed, $\deg p\geq 3$. Let $\beta$ be any root of $p$ of multiplicity at least~2. Let $m_\beta\in \Q[x]$ be the minimal polynomial of $\beta$. Since $m_\beta$ divides $p$ and $p$ is monic with $|p(0)|=1$ we have that $m_\beta\in \Z[x]$. Moreover, since $\beta$ is not a simple root, we have that $m_\beta^2$ divides $p$ (see for instance \cite[Lemma 8.5]{AdBM}). Hence, $p(x)=m_\beta^2(x)q(x)$ for some $q\in \Z[x]$. Thus, $q$ is a monic integer polynomial with $|q(0)|=1$, $\al$ is a root of $q$ and, moreover, it is the only root of $q$ with odd multiplicity. Since $\deg q<\deg p$ the inductive hypothesis ensures that $\al=\pm 1$, and  $\ri$  follows.

The proof of $\rii$ is analogous. 
\end{proof}

\smallskip

Using Lemma \ref{lem:odd} we obtain next a necessary condition on the matrix $A$ as in \eqref{eq:matrixB} so that the simply connected Lie group $G$ associated to $\g=\R e_0 \ltimes_A \R^{4n-1}$ admits lattices.

\begin{proposition}\label{prop:mu=0}
Let $\g=\R e_0 \ltimes_A \R^{4n-1}$ be a $4n$-dimensional almost abelian Lie algebra admitting a hypercomplex structure $\hcx$, with $A$ as in \eqref{eq:matrixA}, for some $B\in \mathfrak{gl}(n-1,\H)$. If the simply connected Lie group $G$ associated to $\g$ has lattices then $\mu =0$ and $\tr B=0$.
\end{proposition}

\begin{proof}
Since $G$ has lattices, it follows from Proposition \ref{prop: Bock} that there exists $t\neq 0$ such that $E:=\e^{tA}$ is conjugate to a matrix in $\operatorname{SL}(4n-1,\Z)$. Hence the characteristic polynomial $p_E$ of $E$ is integer and monic, with $|p_E(0)|=1$. Moreover, it follows from \eqref{eq:matrixA} that $p_E(x)=(x-\e^{t\mu})^3p_F(x)$, where $p_F$ is the characteristic polynomial of $F:=\e^{tB}$. Applying Theorem \ref{thm:Jordan} to the matrix $\sigma(B)$, it follows that the multiplicities of the real roots of $p_F$ are divisible by 4, while the non-real complex roots have even multiplicity. Therefore $e^{t\mu}$ is the only root of $p_E$ with odd multiplicity, thus $e^{t\mu}=1$ according to Lemma \ref{lem:odd}. Since $t\neq 0$ we obtain $\mu=0$. The fact that $\tr B=0$ follows immediately from this and $\tr A=0$.
\end{proof}


We determine next which of the simply connected  unimodular Lie groups whose corresponding Lie algebras appear in Corollary \ref{cor:unimod} admit lattices.

We will denote the simply connected Lie group corresponding to one of the Lie algebras in Theorems \ref{thm: classification1} and \ref{thm: classification2} by replacing $\s$ by $S$; for instance, $S_1^{a,c,d}$ denotes the simply connected Lie group corresponding to $\s_1^{a,c,d}$.

We begin with the nilpotent case. The following result is a consequence of Remark \ref{rem:latt-nilp}.

\begin{proposition}
The nilpotent Lie groups $S_{10}^0$, $S_{16}^0$ and $S_{16}^1$ have lattices.    
\end{proposition}


We state next a non-existence result, which follows easily from Proposition \ref{prop:mu=0}, since the  matrices $A$ defining the corresponding Lie algebras have $\mu\neq 0$.

\begin{proposition}
The following unimodular Lie groups do not admit lattices:
\begin{itemize}
    \item $S_3^{a,b,c,d}$,
    \item $S_4^{a,b,c}$ with $a+c=-\frac34$,
    \item $S_7^{a,c,d}$ with $a+c=-\frac34$, $d>0$,
    \item $S_8^{-\frac74,d}$ with $d>0$,
    \item $S_{11}^{a,c}$ with $a+c=-\frac34$,
    \item $S_{12}^{-\frac74}$,
    \item $S_{14}^{-\frac38,b}$ with $d>0$,
    \item $S_{17}^{-\frac38}$.
\end{itemize}
\end{proposition}


We discuss next all the remaining unimodular Lie groups in Corollary \ref{cor:unimod}, that is, those whose defining matrix $A$ in Theorems \ref{thm: classification1} and \ref{thm: classification2} satisfies $\mu=0$ and $\tr B=0$.

\begin{proposition}\label{prop:latt-g1-2}
The unimodular Lie groups $S_{1}^{a,c,d}$ and $S_{2}^{a,c}$, with $a+c=0$, admit lattices for some values of the parameters, including $S_1^{0,0,d}$ for a countable number of $d>1$ and $S_2^{0,0}$. 
\end{proposition}

\begin{proof}
Let us consider the generic matrix 
\begin{equation}\label{eq:Bcomplex}
    B= \left[\begin{array}{cccc}
    a& -b& &  \\
b&a&& \\
&&a&b \\
&&-b&a\\
    \end{array}\right] \oplus     
    \left[\begin{array}{cccc} 
 -a&-d&& \\
 d&-a&& \\
 &&-a&d \\
 &&-d&-a
\end{array}\right]. 
\end{equation} 
Depending on the choice of the values of $a,b,d$ and setting $A=0_3\oplus B$, the matrix $A$ gives rise to either $\s_{1}^{a,-a,d}$ or $\s_{2}^{a,-a}$. We will work with both cases simultaneously.

Assume first that $a=0$. Then it is well known that choosing $b,d\in \{ \frac{2\pi}{k} \mid k=1,2,3,4,6\}$ the matrix $\e^A$ is either integer or conjugate to an integer matrix. Therefore, the associated simply connected Lie groups admit lattices: if $b\neq d$ we get $S_1^{0,0,d'}$ for some $d'>1$, while if $b=d$ we get $S_2^{0,0}$.

Next we will show that we can also obtain lattices for some values of $a\neq 0$. Let $p\in \Z[x]$ be a monic polynomial of the form 
\begin{equation}\label{eq:complexroots}
    p(x)=x^4-m_3x^3+m_2x^2-m_1x+1,
\end{equation} 
such that all its roots are non-real complex numbers and none of them has modulus one. That is, the roots of $p$ are $\alpha, \bar{\alpha}, \beta, \bar{\beta}$ with $\operatorname{Im}\alpha\neq 0$, $\operatorname{Im}\beta\neq 0$, $|\alpha|=\rho>1$ and $|\beta|=\rho^{-1}<1$ (see Example \ref{ex:p_k}, where we exhibit a family of polynomials satisfying these conditions). We may write then
\[ \alpha=\rho\e^{i\theta}, \quad \beta=\rho^{-1} \e^{i\varphi}, \quad \text{with} \quad \theta,\varphi\in (0,\pi).\]
Let us now consider the matrices 
\[ X_1=
\left[ \begin{array}{cccc} 
 \log \rho &-\theta && \\ \theta &\log \rho && \\ 
 &&-\log \rho &-\varphi \\ && \varphi & -\log \rho \end{array}\right], \quad X_2=\left[ \begin{array}{cccc} 
 \log \rho &\theta && \\ -\theta &\log \rho && \\ 
 && -\log \rho &\varphi \\ && -\varphi & -\log \rho \end{array}\right]. 
 \] 
 Hence, both matrices $\e^{X_1}$ and $\e^{X_2}$ have the polynomial $p$ from \eqref{eq:complexroots} as their characteristic and minimal polynomial. It follows that both $\e^{X_1}$ and $\e^{X_2}$ are conjugate to the companion matrix $C_p$ of $p$, which is clearly in $\operatorname{SL}(4, \Z)$.

 Note that the matrix $X_1\oplus X_2$ is conjugate to the matrix $B$ in \eqref{eq:Bcomplex} with 
 \begin{equation}\label{eq:values}
     a=\log \rho \neq 0, \quad  b=\theta,\quad  d=\varphi.
 \end{equation} 
 It follows that $\e^A$ is conjugate to $0_3\oplus C_p\oplus C_p\in \operatorname{SL}(11,\Z)$ and therefore, for the values given in \eqref{eq:values}, the associated simply connected Lie groups admit lattices. When $\theta\neq \varphi$ we obtain $S_1^{a',-a',d'}$ for certain $a'\in\R$, $d'>1$, whereas when $\theta= \varphi$ we obtain $S_2^{a',-a'}$ for certain $a'<0$. We point out that the condition $\theta= \varphi$ holds if and only if the polynomial $p$ in \eqref{eq:complexroots} is self-reciprocal, that is, when $m_1=m_3$.
\end{proof} 


\begin{example}\label{ex:p_k}
For $k\in\N$ consider the integer polynomial $p_k(x)=x^4-x^3+kx^2-x+1$, which can be written as $p_k(x)=(x-1)^2(x^2+x+1)+kx^2$. Therefore all the roots of $p_k$ are non-real complex numbers and it can be easily seen that for $k\geq 3$ none of them has modulus one. Therefore each polynomial $p_k$ (with $k\in\N,\, k\geq 3$) gives rise to a hypercomplex almost abelian solvmanifold corresponding to some $S_2^{a_k,-a_k}$. 
\end{example}



\begin{proposition}\label{prop:latt-g5}
The unimodular Lie group $S_{5}^{a,-a}$ admits lattices if and only if $a=0$. 
\end{proposition}

\begin{proof}
The matrix $A$ that gives rise to the almost abelian Lie algebra $\s_5^{a,-a}$ is 
\[ A= 0_3\oplus a I_4 \oplus \left[ \begin{array}{cccc} 
 -a & -1 && \\ 1 & -a && \\ 
 && -a & 1 \\ && -1 & -a 
 \end{array}\right].\]
Therefore, for any $t\in \R$ we have
\[ \e^{tA}=I_3\oplus \e^{at}I_4\oplus \e^{-at}\left[ \begin{array}{cccc} 
 \cos t & -\sin t && \\ \sin t & \cos t && \\ 
 && \cos t & \sin t \\ && -\sin t & \cos t 
 \end{array}\right]. \]

 Assume that $S_5^{a,-a}$ admits lattices for $a\neq 0$. Then, according to Proposition \ref{prop: Bock}, there exists $t_0\neq 0$ such that $\e^{t_0 A}$ is conjugate to an integer matrix. Since $at_0\neq 0$, the characteristic polynomial $p$ and the minimal polynomial $m$ of $\e^{t_0 A}$ are given~by 
\begin{gather*}
p(x)=(x-1)^3(x-\e^{at_0})^4  (x^2-2\e^{-at_0}\cos(t_0)x+\e^{-2at_0})^2, \\
m(x)=(x-1)(x-\e^{at_0}) (x^2-2\e^{-at_0}\cos(t_0)x+\e^{-2at_0}). 
\end{gather*} 
Since $m\in \Z[x]$ and $m(0)$ divides $p(0)=-1$ in $\Z$, we have that $m(0)=\e^{-at_0}=1$. Hence $at_0=0$, a contradiction. Therefore, $S_5^{a,-a}$ does not admit lattices for $a\neq 0$.

On the other hand, when $a=0$ it is well known that the matrices $\begin{bmatrix}
\cos t & -\sin t\\ \sin t & \cos t 
\end{bmatrix}$ and $\begin{bmatrix}
\cos t & \sin t\\ -\sin t & \cos t 
\end{bmatrix}$ are conjugate to integer matrices whenever $t\in \{\frac{2\pi}{k} \mid k=1,2,3,4,6\}$. Thus, $S_5^{0,0}$ admits lattices.
\end{proof}


\begin{proposition}
The unimodular Lie group $S_6^0$ admits lattices.
\end{proposition}

\begin{proof}
The matrix $A$ that gives rise to the almost abelian Lie algebra $\s_6^0$ is 
\[ A=\left[\begin{array}{c|c}
          0_3   &  \\
          \hline 
        U   &  0_4
        \end{array}\right] \oplus \left[ \begin{array}{cccc}
 0&-1&& \\ 1&0&& \\  &&0&1\\  &&-1&0
\end{array}\right], \quad\text{ with } U \text{ as in \eqref{eq:matrixU}}.\]
Therefore, for any $t\in \R$ we have
\begin{equation}\label{eq:g6}
\e^{tA}=\left[\begin{array}{c|c}
I_3   &  \\ \hline  tU   &  I_4
\end{array}\right] \oplus \left[ \begin{array}{cccc} 
 \cos t & -\sin t && \\ \sin t & \cos t && \\ 
 && \cos t & \sin t \\ && -\sin t & \cos t 
 \end{array}\right].
\end{equation}  

 The second matrix in \eqref{eq:g6} is conjugate to an integer matrix when $t=\frac{2\pi}{k}$ for $k=1,2,3,4,6$. Clearly, the first matrix in \eqref{eq:g6} is not integer for these values of $t$, but it has the same  Jordan canonical form as the matrix $\left[\begin{array}{c|c}
I_3   &  \\ \hline  U   &  I_4
\end{array}\right]$, which is in $\operatorname{SL}(7,\Z)$. Therefore, $\e^{tA}$ is conjugate to an integer matrix for $t=\frac{2\pi}{k}$, $k=1,2,3,4,6$, and thus $S_6^0$ admits lattices, according to Proposition \ref{prop: Bock}.
\end{proof}


\begin{proposition}
The unimodular Lie group $S_9^{-1}$ has lattices.
\end{proposition}

\begin{proof}
 The matrix $A$ which gives rise to $\s_9^{-1}$ is given by $A= 0_3 \oplus I_4 \oplus (-I_4)$,
and it is conjugate to 
\[ A'=0_3\oplus \begin{bmatrix}
    1 & 0\\ 0 & -1
\end{bmatrix}^{\oplus 4}. \]
For $m\in\N$, $m\geq 3$, let $t_m:=\log \frac{m+\sqrt{m^2-4}}{2}$. Then each $2\times 2$ block in $A'$ generates a block 
$\begin{bmatrix}
    \e^{t_m} & 0 \\ 0 & \e^{-t_m}
\end{bmatrix}$ in $\e^{t_m A'}$, with each such block conjugate to $\begin{bmatrix}
    0 & -1 \\ 1 & m 
\end{bmatrix}\in\operatorname{SL}(2,\Z)$, so that $\e^{t_m A'}$ is conjugate to $I_3\oplus \begin{bmatrix}
    0 & -1 \\ 1 & m 
\end{bmatrix}^{\oplus 4} \in \operatorname{SL}(11,\Z)$.  Therefore $S_9^{-1}$ has lattices, using again Proposition~\ref{prop: Bock}.
\end{proof}


\begin{remark}
The Lie group $S_9^{-1}$ corresponds to the Lie group $G_1$ in \cite[Example 6.4]{AB}.    
\end{remark}


\begin{proposition}
The unimodular Lie group $S_{13}^{0}$ has lattices.
\end{proposition}

\begin{proof}
According to Theorem \ref{thm: classification2}, we have that $\s_{13}^{0}=\R e_0\ltimes_A \R^{11}$ with
\[ A= 0_3 \oplus 
\left[\begin{smallarray}{cccc|cccc} 0& -1& && &&& \\
1&0&&&&&& \\
&&0&1&&&& \\
&&-1&0&&&&\\
\hline
1&&&& 0&-1&& \\
&1&&&1&0&& \\
&&1&& &&0&1 \\
&&&1& &&-1&0
\end{smallarray}\right]. \]
Then, for any $t\in\R$ we have that
\[ \e^{tA}= I_3\oplus \left[\begin{smallarray}{cccc|cccc} \cos t& -\sin t& && &&& \\
\sin t&\cos t&&&&&& \\
&&\cos t&\sin t&&&& \\
&&-\sin t&\cos t&&&&\\
\hline
t\cos t&-t\sin t&&& \cos t&-\sin t&& \\
t\sin t&t\cos t&&&\sin t&\cos t&& \\
&&t\cos t&t\sin t& &&\cos t&\sin t \\ 
&&-t\sin t&t\cos t& &&-\sin t&\cos t
\end{smallarray}\right].\]

Set $s_k=\frac{2\pi}{k}$ for $k\in \{1,2,3,4,6\}$. The Jordan canonical forms of  $\e^{s_1 A}$ and $\e^{s_2 A}$ are respectively given by 
\[ I_3\oplus \left [\begin{smallmatrix}
    1&&&&&&&\\
    1&1&&&&&\\
    &&1&&&&&\\
    &&1&1&&&&\\
    &&&&1&&&\\
    &&&&1&1&&\\
    &&&&&&1&\\
    &&&&&&1&1
\end{smallmatrix} \right ]\quad \text{and} \quad 
I_3\oplus  \left[\begin{smallmatrix}
    -1&&&&&&&\\
    1&-1&&&&&\\
    &&-1&&&&&\\
    &&1&-1&&&&\\
    &&&&-1&&&\\
    &&&&1&-1&&\\
    &&&&&&-1&\\
    &&&&&&1&-1
\end{smallmatrix} \right ],
\]
which are in $\operatorname{SL}(11,\Z)$. 
For $k=3,4,6$ we see that $\e^{s_3 A}$, $\e^{s_4 A}$ and $\e^{s_6 A}$ are respectively conjugate to 
\[ I_3\oplus \begin{bmatrix}
    0&0&0&-1\\ 1& 0& 0&-2\\ 0&1&0&-3 \\ 0&0&1&-2
\end{bmatrix}^{\oplus 2}, \quad I_3\oplus \begin{bmatrix}
    0&0&0&-1\\ 1& 0& 0&0\\ 0&1&0&-2 \\ 0&0&1&0
\end{bmatrix}^{\oplus 2} \quad \text{and} \quad  I_3\oplus \begin{bmatrix}
    0&0&0&-1\\ 1& 0& 0&2\\ 0&1&0&-3 \\ 0&0&1&2
\end{bmatrix}^{\oplus 2},  \]
which are all in $\operatorname{SL}(11,\Z)$. Hence, for each $k\in \{1,2,3,4,6\}$ we have a lattice $\Gamma_k=s_k\Z\ltimes P_k\Z$ in $S_{13}^0$ for some matrix $P_k\in \operatorname{GL}(11,\R)$.
\end{proof}


\begin{remark}
We know from Remark \ref{rem:hkt} that the Lie groups $S_1^{0,0,d}$, $S_2^{0,0}$ and $S_5^{0,0}$ admit a left invariant hyper-Kähler metric and it is well known that they are flat. Moreover, it follows from Propositions \ref{prop:latt-g1-2} and \ref{prop:latt-g5} that they admit lattices (for a countable number of $d>1$ in the case of $S_1^{0,0,d}$). Therefore we obtain 12-dimensional almost abelian solvmanifolds equipped with hyper-Kähler metrics. We recall that the  hypercomplex almost abelian Lie groups with a left invariant HKT non-hyper-Kähler metric are not unimodular and therefore they do not admit lattices (see \cite[Proposition 4.1]{AB}). 
\end{remark}


\section{Hypercomplex almost abelian solvmanifolds arising from integer polynomials}
In this section we introduce a family of integer polynomials and we associate a hypercomplex  solvmanifold $\Gamma_p\backslash G_p$ to each polynomial $p$ in this family (see Proposition \ref{prop:delta_n} below). 
     It turns out that  
    for each $n\geq 2$, there are infinitely many, up to diffeomorphism, $(4n+4)$-dimensional 
    completely solvable almost abelian  hypercomplex solvmanifolds. Moreover, according to  Theorem~\ref{thm:p_or_p*} below,   the assignment $p \mapsto\,  \Gamma_p\backslash G_p$ is in general two-to-one.   

    
For $n\in\N$, $n\geq 2$, let $\Delta_n$ denote the  subset of $\Z[x]$ given by all the polynomials $p\in\Z[x]$ satisfying the following conditions:
\begin{enumerate}
\item[(i)] the degree of $p$ is $n$,
\item[(ii)] $p$ is monic,
\item[(iii)] $p$ has $n$ different real roots, all of them positive, and
\item[(iv)] $p(0)=(-1)^n.$
\end{enumerate}
We will also consider the following distinguished subset of  $\Delta_n$:
\begin{equation}\label{eq:delta-prima}
\Delta_n ':=\{p\in \Delta_n : p(1)\neq 0\}.
\end{equation}
We point out that    $\Delta_n'$ is infinite for $n\geq 2$ (see Lemma \ref{lem:infinite} below). 
The next useful properties are straightforward consequences of the definitions.
\begin{lemma}\label{lem:delta}
\begin{enumerate}\item[]
    \item [$\ri$] For $n\geq 3$, if $p\in \Delta_n$ satisfies $p(1)=0$ then $p(x)=(x-1)\tilde{p}(x)$ with $\tilde{p} \in \Delta_{n-1}'$.
    \item [$\rii$] For $n, m\geq 2$, if $p\in \Delta _n$ and $q\in \Delta _m$ have no common roots then $pq\in \Delta _{n+m}$. 
\end{enumerate}  
\end{lemma}
In order to apply the above lemma, we will make use of  a well known criterion to determine whether two polynomials $p$ and $q$ have common roots,  given  in terms of the  resultant of $p$ and $q$, see \S\ref{subsec:deltan} below, where several additional properties of the polynomials in the set $\Delta_n$, including many examples, are discussed. 


\subsection{Construction of the solvmanifolds associated to  \texorpdfstring{$p\in\Delta_n$}{} }\label{sec:polyn} We will show next how to associate to each $p\in \Delta_n$ a $(4n+4)$-dimensional completely solvable almost abelian solvmanifold admitting an invariant hypercomplex structure. 

Given $p\in \Delta_n$, let $r_1,\ldots, r_n$ denote all its roots, which by assumption are $n$ different positive real numbers. Note that condition (iv) is equivalent to $r_1 r_2\cdots r_n=1$. Assume that $r_1<r_2<\cdots <r_n$ and consider the diagonal matrix
\begin{equation}\label{eq:Xp}
   X_p=\operatorname{diag}(\log r_1,\ldots,\log r_n)\in \mathfrak{sl}(n,\R),   
\end{equation}
and then the diagonal matrix
\begin{equation}\label{eq:A-p}
    A_p=0_3\oplus X_p^{\oplus 4} \in \mathfrak{sl}(4n+3,\R).
\end{equation}  
Let us define $\g_p:=\R e_0\ltimes _{A_p}\R^{4n+3}$, which is a completely solvable almost abelian Lie algebra. Since $A_p$ is in the form \eqref{eq:matrixA} with $B_p$ as in \eqref{eq:matrixB}, we have that $\g_p$ carries a hypercomplex structure. Due to the choice of $p$, the simply connected Lie group $G_p$ associated to $\g_p$ is a semidirect product $G_p=\R\ltimes_\varphi \R^{4n+3}$, where $\varphi\colon\R\to \operatorname{SL}(4n+3,\R)$ is the Lie group morphism given by
\[  \varphi(t)= I_3 \oplus  \operatorname{diag}(r_1^t,\ldots,r_n^t)^{\oplus 4}.\]
Setting $t=1$, we have that the characteristic and the minimal polynomial of the matrix $\e^{X_p}=\operatorname{diag}(r_1,\ldots,r_n)$ are both equal to $p$, therefore $\e^{X_p}$ is conjugate to the companion matrix $C_p$ of $p$. It follows from  $p\in\Delta_n$ that $C_p\in \operatorname{SL}(n,\Z)$, and hence $\varphi(1)=\e^{A_p}$ is conjugate to the matrix  $\widetilde{C}_p\in \operatorname{SL}(4n+3,\Z)$ given by $\widetilde{C}_p=I_3\oplus C_p^{\oplus 4}$. Thus, according to Proposition \ref{prop: Bock}, $G_p$ admits a lattice $\Gamma_p:=\Z\ltimes_{\varphi(1)} Q_p \Z^{4n+3}$, where $Q_p\in \operatorname{GL}(4n+3,\R)$ satisfies $Q_p^{-1}\varphi(1) Q_p=\widetilde{C}_p$. This lattice is isomorphic to $\Z\ltimes_{\widetilde{C}_p} \Z^{4n+3}$, which is independent of $Q_p$.

\medskip 

To sum up, we have

\begin{proposition}\label{prop:delta_n}
Each polynomial $p\in \Delta_n$ gives rise to a $(4n+4)$-dimensional hypercomplex almost abelian solvmanifold $\Gamma_p\backslash G_p$, with $G_p$ completely solvable.
\end{proposition}

\begin{remark}
A  simpler analogue of the previous construction allows us to obtain almost abelian solvmanifolds equipped with invariant complex structures. Indeed, set  $\tilde{A}_p=0_1\oplus X_p^{\oplus 2}\in \mathfrak{sl}(2n+1,\R)$, where $X_p$ is as in \eqref{eq:Xp}, and define $\tilde{\g}_p:=\R \ltimes_{\tilde{A}_p} \R^{2n+1}$. It follows from \cite{LRV} that $\tilde{\g}_p$ admits a complex structure and moreover, the same ideas from the previous paragraphs show that $\tilde{G}_p$, the simply connected Lie group associated to $\tilde{\g}_p$, has lattices.
\end{remark}


The next result shows that the polynomials in $\Delta_n'$ are more interesting than those having~$1$~as a root. Recall from Lemma \ref{lem:delta}$\ri$ that if  $p\in \Delta_n$, $n\geq 3$, satisfies $p(1)=0$ then  $p(x)=(x-1)\tilde{p}(x)$ with $\tilde{p}\in \Delta_{n-1}'$. 
\begin{proposition}\label{prop:torus}
    For $n\geq 3$, let  $p\in \Delta_n$ such that $p(1)=0$, and let   $\tilde{p}\in \Delta_{n-1}'$ as above.  Then $\Gamma_{p}\backslash G_{p}$  is diffeomorphic to $\Gamma_{\tilde{p}}\backslash G_{\tilde{p}} 
    \times \mathbb{T}^4$, where $\mathbb{T}^4$ denotes the $4$-dimensional torus. Moreover, this diffeomorphism is hyperholomorphic, where the latter manifold is equipped with the product hypercomplex structure. 
\end{proposition}
\begin{proof} It follows from \eqref{eq:Xp} and \eqref{eq:A-p} that $A_p\in \g\l(4n+3,\R)$ is conjugate to $A_{\tilde p}\oplus 0_4$, so that $G_{p}$ is isomorphic to $G_{\tilde p}\times \R^4$ via a hyperholomorphic map. Furthermore, the corresponding lattice $\Gamma_{ p}$ in $G_{ p}$ is given by $\Gamma_{ p}=\Z\ltimes Q_{ p}\Z^{4n+3}$, where $Q_{ p}$ can be chosen as $Q_{\tilde p}\oplus I_4$. Hence, $\Gamma_{ p}=\Z\ltimes (Q_{\tilde p} \Z^{4n-1}\times \Z^4)\cong \Gamma_{\tilde p}\times \Z^4$, and the proposition follows from Theorem \ref{thm:solv-isom}. 
    \end{proof}


For polynomials $p$ and $q$ as in Lemma \ref{lem:delta}$\rii$, we observe that the solvmanifold associated to $pq$ can be hyperholomorphically embedded in the product $\Gamma_{p}\backslash G_{p}\times \Gamma_{q}\backslash G_{q}$. 
\begin{proposition}\label{prop:hcx_sub}
For $n,m\geq 2$, if $p\in \Delta_n$ and $q\in \Delta_m$ have no common roots then $\Gamma_{pq}\backslash G_{pq}$ is a codimension $4$ hypercomplex submanifold of $\Gamma_{p}\backslash G_{p}\times \Gamma_{q}\backslash G_{q}$, where the latter is equipped with the product hypercomplex structure.
\end{proposition}
\begin{proof}
 Let us denote $\g_p=\R e_0\ltimes_{A_p} \R^{4n+3}$, $\g_q=\R f_0\ltimes_{A_q} \R^{4m+3}$, and $e_\alpha= J^p_\alpha e_0$, $f_\alpha= J^q_\alpha f_0$, $\alpha=1,2,3$, where $\{J_\al^p\}$ and $\{J_\al^q\}$ are the hypercomplex structures on $\g_p$ and $\g_q$, respectively. Moreover, let us denote $\h_p\cong \R^{4n}$ and $\h_q\cong \R^{4m}$ the corresponding subspaces  given by Theorem \ref{thm:characterization}. Set $\a:= \text{span}\{e_1+f_1,e_2+f_2,e_3+f_3\}$ and let $\g$ denote the following subspace of $\g_p\times  \g_q$:
\[ \g=\R(e_0+f_0)\oplus \a \oplus (\h_p\oplus \h_q).\]
We point out that $\g$ is invariant by the  hypercomplex structure $\{J_\al^p\oplus J_\al^q\}_{\al=1,2,3}$ on $\g_p\times  \g_q$. 
Note that $\R(e_0+f_0)\oplus \a$ is an abelian subalgebra of $\g_p\oplus \g_q$, and the action of $e_0+f_0$ on $\h_p\oplus \h_q$ is given by
\[ [e_0+f_0,x+y]=X_p^{\oplus 4} x+X_q^{\oplus 4} y, \quad x\in \h_p, \, y\in \h_q. \] 
Since $[e_\al+f_\al,\h_p\oplus \h_q]=0$ for $\al=1,2,3$, we have that $\g$ is again almost abelian, with codimension 1 abelian ideal $\u:=\a \oplus (\h_p\oplus \h_q)$. Moreover, the action of $e_0+f_0$ on $\u$ is given by 
\[ \ad_{(e_0+f_0)}|_\u= 0_3\oplus X_p^{\oplus 4}\oplus X_q^{\oplus 4}.\]
Since this matrix is conjugate to the matrix $A_{pq}$ as in \eqref{eq:A-p} with corresponding  polynomial $pq$, we have that $\g\cong \g_{pq}$, and this isomorphism preserves the corresponding hypercomplex structures. 

The Lie groups $G_p$ and $G_q$ can be decomposed as 
\[ G_p= \R \ltimes _{\varphi_p} (\R^3\times \R^{4n}),\qquad  
G_q= \R \ltimes _{\varphi_q} (\R^3\times \R^{4m}),\]
with
\[  \varphi_p(t)= I_3 \oplus  \operatorname{diag}(r_1^t,\ldots,r_n^t)^{\oplus 4}, \qquad 
  \varphi_q(t)= I_3 \oplus  \operatorname{diag}(s_1^t,\ldots,s_m^t)^{\oplus 4},\]
where $r_1, \dots ,r_n$ are the roots of $p$ and $s_1,\dots , s_m $ are the roots of $q$. Analogously, $G_{pq}$ decomposes~as 
\[
G_{pq}= \R \ltimes _{\varphi_{pq} }(\R^3\times \R^{4n}\times \R^{4m}), \qquad 
\varphi_{pq}(t)= I_3 \oplus  \operatorname{diag}(r_1^t,\ldots,r_n^t)^{\oplus 4} \oplus  \operatorname{diag}(s_1^t,\ldots,s_m^t)^{\oplus 4}.
\]
We may consider $G_{pq}$ as a Lie subgroup of $G_p\times G_q$ via the following map $G_{pq} \to G_p\times G_q$:
\[  (t, x,v,w)\mapsto ((t,x,v),(t,x, w)), \quad t\in \R,\, x\in \R^3, \, v\in\R^{4n} ,\, w\in \R^{4m}.
\]Under this map, the lattice $\Gamma_{pq}$ can be considered as a subgroup of $\Gamma_p\times \Gamma_q$. 
Therefore, $\Gamma_{pq}\backslash G_{pq}$ is a submanifold of $\Gamma_{p}\backslash G_{p}\times \Gamma_{q}\backslash G_{q}$, which is clearly of codimension 4. Moreover, the inclusion is a hypercomplex map.
\end{proof}


According to Proposition \ref{prop:delta_n}, there is a natural map  which associates to each $p\in \Delta_n$ a hypercomplex solvmanifold $\Gamma_p\backslash G_p$ of dimension $4n+4$. This map is not one-to-one, moreover, it is in general two-to-one, as Theorem \ref{thm:p_or_p*} below  shows (compare with \cite[Theorem 2.5]{Hu}).

In order to state the theorem, we associate to each polynomial $p$ a polynomial $p^*$  defined by
\[ p^*(x)=(-1)^nx^np(x^{-1}).  \]
Up to sign, $p^*$ is the reciprocal polynomial of $p$ and it is monic precisely when  $p(0)= (-1)^n.$  It follows that $(p^*)^*=p$, and 
$p\in \Z[x]$ if and only if $p^*\in\Z[x]$. We note that:
\begin{enumerate}
    \item if $p\in \Delta_n$ then $p^*\in\Delta_n$. Indeed, if $r_1,\ldots, r_n$ are the distinct positive roots of $p$, with $r_1 \cdots r_n =1$, then $r_1^{-1},\ldots, r_n^{-1}$ are distinct positive roots of $p^*$ and $r_1^{-1}\cdots r_n^{-1}=1$,  therefore, $p^*\in \Delta_n$; 
    \item if $p\in\Delta_n$ is the characteristic polynomial of a matrix $Q\in\operatorname{SL}(n,\Z)$ then $p^*$ is the characteristic polynomial of $Q^{-1}$. 
\end{enumerate}


\begin{theorem}\label{thm:p_or_p*}
 Let $p,q\in \Delta_n$, with associated solvmanifolds $\Gamma_p\backslash G_p$ and $\Gamma_q\backslash G_q$. Then $\Gamma_p\backslash G_p$ and $\Gamma_q\backslash G_q$ are diffeomorphic if and only if $q=p$ or $q=p^*$.    
\end{theorem}

\begin{proof}
We already know that $p^*\in \Delta_n$. 
Let us show next that $\Gamma_{p^*}\backslash G_{p^*}$ is diffeomorphic to $\Gamma_{p}\backslash G_{p}$. Recall that $\Gamma_{p^*}\cong \Z \ltimes_{\widetilde{C}_{p^*}} \Z^{4n+3}$, where $\widetilde{C}_{p^*}=I_3\oplus C_{p^*}^{\oplus 4} $ and $C_{p^*}$ is the companion matrix of $p^*$. Since $p^*$ is the characteristic polynomial of both $(C_p)^{-1}$ and $C_{p^*}$ and coincides also with the minimal polynomial of these matrices, we have that$(C_p)^{-1}$  is conjugate   to $C_{p^*}$ over $\Q$. But, if two integer matrices are conjugate  over $\Q$, they are conjugate over $\Z$, according to \cite[page 75]{Ka}. Hence, there exists $P\in\operatorname{GL}(n,\Z)$ such that $(C_p)^{-1}=PC_{p^*}P^{-1}$. This implies that $\Z\ltimes_{C_{p^*}} \Z^n\cong \Z\ltimes_{(C_p)^{-1}} \Z^n$, via the isomorphism $(k,v)\mapsto (k,Pv)$. Therefore
\[ \Gamma_{p^*}\cong \Z \ltimes_{\widetilde{C}_{p^*}} \Z^{4n+3}\cong \Z \ltimes_{(\widetilde{C}_p)^{-1}} \Z^{4n+3}.\]
However, it is well known that the semidirect products $\Z\ltimes_Q \Z^m$ and $\Z\ltimes_{Q^{-1}} \Z^m$ are isomorphic for any $Q\in \operatorname{SL}(m,\Z)$, via the isomorphism $(k,v)\mapsto (-k,v)$. As a consequence, we have that
\[ \Gamma_{p^*} \cong \Z \ltimes_{(\widetilde{C}_p)^{-1}} \Z^{4n+3} \cong \Z \ltimes_{\widetilde{C}_p} \Z^{4n+3}\cong \Gamma_p.\]
This implies that $\Gamma_{p^*}\backslash G_{p^*}$ and $\Gamma_{p}\backslash G_{p}$ are diffeomorphic, according to Theorem \ref{thm:solv-isom}.

Conversely, assume now that $\Gamma_p\backslash G_p$ and $\Gamma_q\backslash G_q$ are diffeomorphic, where $G_p=\R\ltimes_{\varphi}\R^{4n+3}$ and $G_q=\R\ltimes_{\psi}\R^{4n+3}$, with $\varphi(t)=\e^{tA_p}$, $\psi(t)=\e^{tA_q}$, $t\in \R$. Then $\Gamma_p$ and $\Gamma_q$ are isomorphic and since both $G_p$ and $G_q$ are completely solvable, Theorem \ref{thm:Saito} ensures that the isomorphism between the lattices extends to a Lie group isomorphism between $G_p$ and $G_q$. That is, there exists an isomorphism $F:G_p\to G_q$ such that $F(\Gamma_p)=\Gamma_q$ and $f:=F|_{\Gamma_p}:\Gamma_p \to \Gamma_q$ is the given isomorphism. Since we have that the Lie algebras $\g_p$ and $\g_q$ are isomorphic, it follows from Lemma \ref{lem:ad-conjugated} that there exists a non-zero $c\in \R$  such that $A_p$ and $cA_q$ are conjugated in $\operatorname{GL}(4n+3,\R)$. According to  Theorem \ref{thm:Lie-iso}$\rii$, the isomorphism $F$ is given by $F(t,v)=(\mu  t, g(t,v) )$ for $\mu=\pm c$ and some $C^\infty$ function $g:G_p\to \R^{4n+3}$. 

Recall the explicit descriptions of the lattices $\Gamma_p$ and $\Gamma_q$: 
$\Gamma_p=\Z\ltimes_{\varphi(1)} Q_p \Z^{4n+3}\subset G_p$  and $\Gamma_q=\Z\ltimes_{\psi(1)} Q_q \Z^{4n+3}\subset G_q$, where $Q_p, \, Q_q\in \operatorname{GL}(4n+3,\R)$ satisfy $Q_p^{-1}\varphi(1) Q_p=\widetilde{C}_p$ and $Q_q^{-1}\psi(1) Q_q=\widetilde{C}_q$. We have that $F(\Gamma _p)= \mu \Z\times \Lambda_p\subset \Gamma_q$, where $\Lambda _p=g(\Gamma _p)$, therefore, $\mu \Z =\Z$ due to $F(\Gamma_p)=\Gamma_q$, which implies that $\mu =\pm 1$, hence $c=\pm 1$.

If $c=1$ then $\varphi(1)=\e^{A_p}$ and $\psi(1)=e^{A_q}$ are conjugate, so that they have the same characteristic polynomial, that is, $(x-1)^3\,q(x)^4=(x-1)^3\,p(x)^4$, which implies $p=q$ since $p,\, q\in \Delta_n$.  

On the other hand, if $c=-1$ then $A_q$ is conjugate to $-A_p$, so  that $\e^{A_q}$ is conjugate to $(\e^{A_p})^{-1}$. Therefore, the  characteristic polyomials of $\e^{A_q}$ and  $(\e^{A_p})^{-1}$ coincide, which implies that  $q=p^*$. This concludes the proof. 
\end{proof}


\begin{example}
There is a way to produce an infinite number of polynomials in $\Delta_n$, beginning with a given $p\in\Delta_n$.  Indeed, if $p$ has roots $r_1,\ldots, r_n$ consider, for 
$k\in \Z, \, k\neq 0$, the polynomial $p_k$ given by
\[ p_k(x)=(x-r_1^k)\cdots (x-r_n^k).\] 
Using Newton's identities it can be seen that $p_k\in \Z[x]$ for $k\geq 2$ and, since $p_{-1}=p^*$, this also holds for $k<0$. Moreover, $p_k$ has  positive distinct  roots  $r_1^k, \ldots, r_n^k$, therefore,  $p_k\in \Delta_n$. 

We observe that the corresponding matrix $A_{p_k}$, as in \eqref{eq:A-p}, is given by
$A_{p_k}=k A_p$. 
It follows from Lemma \ref{lem:ad-conjugated} that the associated almost abelian Lie groups $G_p$ and $G_{p_k}$ are isomorphic for any $k\neq 0$. Moreover, according to Theorem \ref{thm:p_or_p*} we have that for $j,k\neq 0$, the corresponding solvmanifolds $\Gamma_{p_j}\backslash G_{p_j}$ and $\Gamma_{p_k}\backslash G_{p_k}$ are diffeomorphic if and only if $k=\pm j$.
\end{example}


\subsection{{Properties of the polynomials in  \texorpdfstring{$\Delta_n$}{} }} \label{subsec:deltan} In order to obtain explicit examples of the construction above we will prove several properties of the polynomials in the set $\Delta_n$ and we will  exhibit infinite subsets of $\Delta_n$ for each $n\geq 2$.

\begin{lemma}\label{lem:combinatorio}
Let $p\in \Delta_n$ be given by
\[ p(x)=x^n+\sum_{j=1}^{n-1}(-1)^j m_{n-j}x^{n-j} + (-1)^n.\]
Then $m_j>\binom{n}{j}$ for all $j=1,\ldots,n-1$. 
\end{lemma}

\begin{proof}
Let $r_1,\ldots,r_n$ be the $n$ different real roots of $p$, where $r_j>0$ for all $j$. It follows that 
\[ p(x)=(x-r_1)\cdots (x-r_n).\]
Note that $r_1\cdots r_n=1$. Then, expanding this product, we obtain that the coefficient of the monomial $x^{n-j}$ is
\[ (-1)^j m_{n-j} = (-1)^{j} \sum_{k_1<\cdots < k_{j}} r_{k_1}\cdots r_{k_{j}},\]
so that,  for all $j$,
\begin{equation}\label{eq:mj}
m_{n-j}=\sum_{k_1<\cdots < k_{j}} r_{k_1}\cdots r_{k_{j}} >0.
\end{equation} 

Next, for $1\leq j\leq n-1$  we use the inequality of arithmetic and geometric means to the $\binom{n}{j}$ positive real numbers $r_{k_1}\cdots r_{k_{j}}$, $1\leq k_1<\cdots < k_{j}\leq n$. Using \eqref{eq:mj} we obtain that the arithmetic mean is $\binom{n}{j}^{-1}m_{n-j}$, whereas using $r_1\cdots r_n=1$ we get that the geometric mean is~$1$. Thus $m_{n-j}\geq \binom{n}{j}$. If we had $m_{n-j}=\binom{n}{j}$ for some $j$ then all the numbers $r_{k_1}\cdots r_{k_{j}}$ would be equal to $1$, and this would imply that $r_1=\cdots=r_n$, which is a contradiction. Therefore, for all $j$ we have
\[ m_{n-j} > \binom{n}{j} =\binom{n}{n-j} ,\]
 and the proof is complete.
\end{proof}


\begin{example}\label{ex:delta2}
 For each $m\in \N, \, m\geq 3$, let $t_m=\frac{m+\sqrt{m^2-4}}{2}>0$, then $t_m^{-1}=\frac{m-\sqrt{m^2-4}}{2} \neq t_m$. We define   the following  quadratic  polynomial $h_m$: \[ h_m(x)=\left(x-t_m^{-1}\right)(x-t_m)=x^2-mx+1.\] It is easy to show that 
\[
\Delta_2=\{h_m : m\in \N, \, m\geq 3\}. 
\]
Note that  $h_m(1) \neq 0$ for all $m\geq 3$, therefore, $\Delta_2'=\Delta_2 $. Moreover, $h_m$ and $h_n$ have no common roots for $m\neq n$. 
\end{example}


\begin{example}\label{ex:cubic}
Let us consider a cubic monic integer polynomial given by 
\[ f_{m,n}(x)=x^3-mx^2+nx-1, \quad m,n\in \N.\]
Note that any real root of $f_{m,n}$ is positive, since $f_{m,n}(x)\leq -1 $  for $x\leq 0$. Moreover, all the roots of $f_{m,n}$ are distinct  real numbers if and only if its discriminant $D(f_{m,n})$ is positive, where $D(f_{m,n})=m^2n^2-4m^3-4n^3+18mn-27$. Therefore, $f_{m,n}\in\Delta_3$ if and only if $D(f_{m,n})>0$. In particular,  we obtain from Lemma \ref{lem:combinatorio} that $m>3$ and $n>3$. Therefore, the set $\Delta_3$ is in one-to-one correspondence with the integer points in the interior of the region in Figure \ref{fig:region}. Note that $f_{m,n}\in \Delta_3'$ if and only if $m\neq n$. 

\smallskip

\begin{figure}[ht]
\includegraphics[scale=0.5]{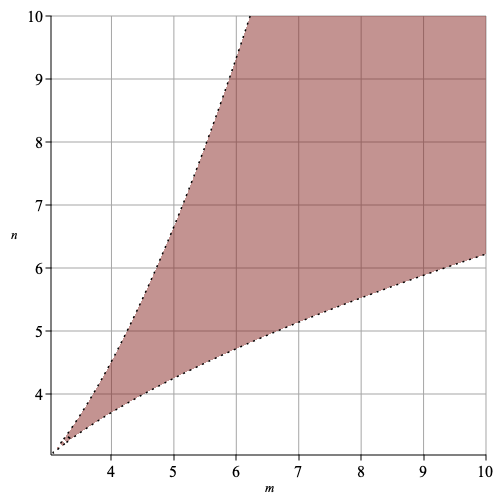}
\caption{Discriminant of $f_{m,n}>0$}
\label{fig:region} 
\end{figure} 
\end{example}


\begin{proposition}\label{prop:positive}
Let $p\in\Z[x]$ be a monic polynomial of degree $n\geq 2$ of the form
\[ p(x)=x^n+\sum_{j=1}^{n-1}(-1)^j m_{n-j}x^{n-j} + (-1)^n, \]
with $m_j>0$ for all $j=1,\ldots,n-1$, and set $m_0=m_n=1$. If 
\begin{equation}\label{eq:desigualdad}
m_j^2-4m_{j-1}m_{j+1} >0, \quad j=1,\ldots,n-1,
\end{equation} 
then $p\in \Delta_n$. 
\end{proposition}

\begin{proof}
We observe that $p(-x)=(-1)^n q(x)$, where $q(x)$ is given by \[ q(x)= x^n+\sum_{j=1}^{n-1} m_{n-j}x^{n-j} +1 .\] Since \eqref{eq:desigualdad} is satisfied, it follows from \cite[Theorem 1]{Ku} that $q$, and hence $p$, has $n$ distinct real roots. The roots of $q$ are negative, since clearly $q(x)> 0$ for $x\geq 0$, which means that  all roots of $p$ are positive, that is, $p\in \Delta_n$. 
%
%
%
\end{proof}


\begin{example}
For $m,n, r\in\N$, consider the integer polynomial given by 
\[ p(x)=x^4-mx^3+nx^2-rx+1.\]
It is easily verified that the conditions in Proposition \ref{prop:positive} are fulfilled if and only if 
\[ 2\sqrt{mr}<n<\frac 14 \min\{ m^2, r^2\}.\]
Therefore, for each choice of  $m>8$ and $r>8$, we find a finite number of values of $n$ such that $p\in \Delta_4$.  Each $p$ gives rise to a 20-dimensional hypercomplex almost abelian solvmanifold.

In a similar fashion, the integer polynomial 
\[ q(x)=x^5-mx^4+nx^3-rx^2+sx-1\]
satisfies the conditions in Proposition \ref{prop:positive} if and only if
\[ n<\frac{m^2}{4}, \quad r<\frac{s^2}{4}, \quad 4mr<n^2, \quad 4sn<r^2.\]
Thus, for each choice of $m> 16$ and $s>16$ we find a finite number of values of $n$ and $r$ such that $q\in \Delta_5$. Each $q$ gives rise to a 24-dimensional hypercomplex almost abelian solvmanifold. For instance, if $m=s=17$ then we can choose $n,r\in\{69,70,71,72\}$; moreover, if $n\neq r$ then $q\in\Delta_5'$. 
\end{example}


In order to prove that the set $\Delta_n'$ defined in \eqref{eq:delta-prima} is infinite for $n\geq 2$, 
we will apply  Lemma \ref{lem:delta}$\rii$ and Examples \ref{ex:delta2} and \ref{ex:cubic}.  
 We recall next the following useful criterion to determine whether two polynomials have common roots.
If $p, \, q\in \R[x]$ have only real roots, then they have no common roots  if and only if $\operatorname{Res}(p,q)\neq 0$, where  $\operatorname{Res}(p,q)$ is the resultant of $p$ and $q$, defined by $\operatorname{Res}(p,q)=\det (\mathrm{Syl}(p,q))$. Here $\mathrm{Syl}(p,q)$ denotes the Sylvester matrix associated to $p$ and $q$, defined in the following way: if $p$ has degree $m$ and $q$ has degree $n$, with $p(x)= \sum _{i=0}^m a_i x^i , \, q(x)= \sum _{j=0}^n b_j x^j$, then $\mathrm{Syl}(p,q)$ is the $(m+n)\times(m+n)$ matrix given by:
\[
	\mathrm{Syl}(p,q)=\left[\phantom{\begin{matrix}a_0\\ \ddots\\a_0\\b_0\\ \ddots\\ \ddots \\b_0 \end{matrix}}
	\right.\hspace{-1.5em}
	\begin{matrix}
		a_m & a_{m-1} &\cdots& \cdots & a_0 &  & \\
		&  	\ddots& \ddots& & & \ddots & \\
		& & a_m & a_{m-1} & \cdots & \cdots &a_0 \\
		b_n & b_{n-1} & \cdots & b_0 & & & \\
		&\ddots &\ddots & & \ddots &  &\\ 
  & &\ddots & \ddots & &\ddots  &\\
&	& &	  b_n & b_{n-1}& \cdots & b_0
	\end{matrix}
	\hspace{-1.5em}
	\left.\phantom{\begin{matrix}a_0\\ \ddots\\a_0\\b_0\\ \ddots\\ \ddots \\ b_0 \end{matrix}}\right]\hspace{-1em}
	\begin{tabular}{l}
		$\left.\lefteqn{\phantom{\begin{matrix} a_0\\ \ddots\\ a_0\ \end{matrix}}}\right\}n$\\
		$\left.\lefteqn{\phantom{\begin{matrix} b_0\\ \ddots  \\\ddots    \\ b_0\ \end{matrix}}} \right\}m$
	\end{tabular}
	\]
 where the $(i,j)$ coefficient is  zero in the following cases: 
\[ 1<i,\, j<n, \qquad i<n, \, j> m+1,\qquad i> n+1, \, j<m, \qquad   i,\,  j> n+1 .\]


\begin{lemma}\label{lem:infinite} 
The set $\Delta_n '$ is infinite for $n\geq 2$.
\end{lemma}

\begin{proof}
 Assume first that $n$ is even, $n=2k$, $k\geq 1$, and let $m_1, \dots , m_k\geq 3$ be integers, which are chosen to be distinct if $k>1$. Then it follows from Lemma \ref{lem:delta}$\rii$ that $\Pi _{j}h_{m_j} \in \Delta_n'$, where $h_{m_j}$ are the quadratic polynomials from Example \ref{ex:delta2}. 


The assertion in the statement  is true for   $n=3$ (see Example \ref{ex:cubic}).  Consider next the case when $n$ is odd, $n\geq 5$, that is, $n=2k+3$ with $k\geq 1$. Let $m_1, \dots , m_k\geq 4$ be integers, which are chosen to be distinct if $k>1$. Then it follows from Lemma \ref{lem:delta}$\rii$ that $h:=\Pi _{j}h_{m_j} \in \Delta_{2k}$, where $h_{m_j}$ are the quadratic polynomials from Example \ref{ex:delta2} (note that  $h\in \Delta_{2k}'$). Let $f_{6,7}\in \Delta_3$ from Example \ref{ex:cubic}, then  $f_{6,7}\in \Delta_3'$. We will show next that $f_{6,7}$ and $h$  have no common roots, therefore, it will follow from Lemma \ref{lem:delta}$\rii$ that $q:=f_{6,7}\, h\in \Delta_n'$. To prove that $f_{6,7}$ and $h$ have no common roots it suffices to show that $f_{6,7}$ and $h_{m_j}$ have no common roots for all $j$. Indeed, for each $m\geq 4$,  
 we compute the resultant $\operatorname{Res}( h_{m}, f_{6,7})$:
 \[\operatorname{Res}( h_{m}, f_{6,7})= \det \begin{bmatrix}
     1&-m&1&0&0 \\
     0&1&-m&1&0 \\
     0&0&1&-m&1 \\
     1&-6&7&-1&0\\
     0&1&-6&7&-1
 \end{bmatrix}= -m^3+13m^2-52m+61. \] 
 It can be easily seen that $\operatorname{Res}( h_{m}, f_{6,7})<0$ for $m\geq 4$, therefore, $f_{6,7}$ and $h_{m}$ have no common roots, and the lemma follows.
\end{proof}


\section{Appendix A}\label{app:complex-nilp}
In this section we state, without proof, the analogue of Theorem \ref{thm:classif_nilp} for nilpotent  almost abelian Lie algebras admitting a complex structure, which follows  by applying the usual Jordan normal form. 

Given a complex structure $J$ on $\R^{2q}$ we denote by 
\[ \g\l (q,\C )=\{T\in \g\l (2q,\R ) : TJ=JT\}.
\]

We recall from \cite[Lemma 3.1]{AO} (see also \cite[Lemma 6.1]{LRV}) the characterization of almost abelian Lie algebras with a Hermitian structure: 
\begin{lemma}\label{lem:hermitian}
Let $\g$ be an almost abelian Lie algebra with codimension one abelian ideal $\u$, admitting  a Hermitian structure  $(J, \pint)$. Then $\a:=\u \cap J\u $ is a $J$-invariant abelian ideal  of codimension $2$. Moreover, there exist an orthonormal basis $\{f_1,f_2=Jf_1\}$ of ${\a}^{\perp}$, $v_0\in\a$ and $\mu\in\R$ such that $f_2 \in \u$, $[f_1,f_2]=\mu f_2 + v_0$ and $\ad_{f_1}|_{\a}$ commutes with $J|_{\a}$.
\end{lemma}

It follows from Lemma  \ref{lem:hermitian}  that  $\g$ can be written as $\g=\R e_0\ltimes_\cA \R^{2n-1}$, where the matrix $\cA\in\mathfrak{gl}(2n-1,\R)$ defined by the adjoint action of $e_0$ on $\R^{2n-1}$ has the following expression in a basis $\{e_1\} \cup \cC$ of $\u$,  where $\cC$ is a basis of $\a$:
\begin{equation}\label{eq:matrixA-complex} 
\cA=\left[
	\begin{array}{c|ccc}      
		\mu  & 0 &\cdots  & 0\\
		\hline
		 |& & & \\
		  v & & \cB &\\
		 | & & &
	\end{array}
	\right], \qquad \mu \in \R, \; v\in \R^{2n-2}, \; \cB\in\mathfrak{gl}(n-1,\C)\subset \mathfrak{gl}(2n-2,\R).
\end{equation} 
If the matrix of $J$ in the basis $\mathcal C$ is given by $\begin{bmatrix}0&-I \\
    I&0
\end{bmatrix}$, then the corresponding matrices in $\mathfrak{gl}(n-1,\C )$ take  the form $\begin{bmatrix}X&-Y \\
    Y&X
\end{bmatrix}$ with $X,Y\in M_{n-1}(\R )$. 
We can identify $\mathfrak{gl}(n-1,\C)\subset \mathfrak{gl}(2n-2,\R)$ with $M_{n-1}(\C)$ via the $\R$-algebra isomorphism:  \[  \varphi\colon \mathfrak{gl}(n-1,\mathbb{C}) \to M_{n-1}(\C), \qquad   
    \begin{bmatrix} X& -Y\\
Y&X
\end{bmatrix}\mapsto X+iY .\]
In order to carry out the study of  nilpotent  almost abelian  Lie algebras admitting a complex structure, let $\g=\R \ltimes  _\cA \R ^{2n-1}$ where $\cA$ is  as in \eqref{eq:matrixA-complex} with $\mu=0$ and $\cB\in \g\l (n-1, \C)$, $\cB$ nilpotent. 
If  $\cB\neq 0$,  
the Jordan normal form of $\varphi (\cB)$ is given by: 
\begin{equation}   \label{eq:Jordan_nilp-complex} 
  (j_{m_1})^{\oplus p_1}\oplus \cdots \oplus  (j_{m_r})^{\oplus p_r}\oplus { 0}_s, \quad m_1 > \dots >m_r\geq 2, \;s\geq 0, \; p_k>0,
\end{equation}
for all $k$, where $0_s$ is the zero $s\times s$ matrix and $j_{m_k}$ is an elementary Jordan block as in \eqref{eq:matrixjm}. 
We will encode all this data associated to $\cB$ using the following notation:
\[ {\Sigma}'(\cB):=(r,m_1,\ldots, m_r, p_1,\ldots, p_r, s),  \quad \;\; m_1 > \dots >m_r\geq 2, \;s\geq 0, \; p_k>0 .  \]
We point out that \begin{equation}\label{eq:dimn-complex}
    n-1=\sum_{i=1}^r m_i p_i +s .
\end{equation}
If $\cB=0$ we have $r=0, \; s=n-1$. 

In order to state Theorem \ref{thm:classif_nilp-complex} below, we introduce first some notation. Consider the following nilpotent matrices:
\begin{equation}\label{eq:matrixN-complex}
      \cN_\ell= \left[ \begin{array}{c|ccc}
    0& & &  \\
    \hline 
    \begin{array}{c}
    1  \\ 
    0\\
    \vdots \\
    0 
    \end{array}& &\mathcal \cJ_{m_\ell} &
\end{array}\right] \oplus (\cJ _{m_\ell})^{\oplus (p_\ell-1)}, \; 1\leq \ell\leq r, \quad 
\cN = \left[ \begin{array}{c|ccc}
    0&  & & \\
    \hline 
    \begin{array}{c}
    1  \\ 
    0\\
    \vdots \\
    0 
    \end{array}&  & 0_{2s}&   
\end{array}\right], \; s>0,
\end{equation}
where  $\cJ_{m_\ell}$ is  the following  matrix: 
\[\cJ_{m_\ell}= \left[
\begin{array}{ccccc}
   0_2  &  0_2 & \cdots & \cdots & 0_2 \\
   I_2 & 0_2 & 0_2 & \cdots & 0_2 \\
   0_2 & I_2 & 0_2 & \cdots & 0_2 \\
   \vdots & \vdots & \ddots & \ddots & 0_2 \\
   0_2 & 0_2 & 0_2 & I_2 & 0_2
\end{array}\right] \in M_{2m_\ell}(\R).
\]
Here, $0_2$ and $I_2$ are the $2\times 2$ zero and identity matrices, respectively. Note that $\cJ_{m_\ell}$ is conjugate to $\left(j_{m_\ell}\right)^{\oplus 2}$. 
Let
\begin{equation}
     \label{eq:Al-complex} 
     \begin{split} 
     \cA_0&=0_1 \oplus ({ \cJ _{m_1}})^{\oplus p_1}\oplus \cdots \oplus  ( \cJ _{m_r})^{\oplus p_r} 
\oplus 0_{2s}, \\
     \cA_\ell&= \bigoplus_{i=1}^{\ell-1}({ \cJ_{m_i}})^{\oplus p_i}\oplus   \; \cN_\ell \; \oplus\bigoplus_{i=\ell+1}^{r} ( \cJ _{m_{i}})^{\oplus p_{i}}\oplus  { 0}_{2s}, \quad 1\leq \ell\leq r,\\
  \cA_{r+1}&= ({ \cJ _{m_1}})^{\oplus p_1}\oplus \cdots \oplus  ( \cJ _{m_r})^{\oplus p_r}\oplus \cN .
\end{split} \end{equation}
We point out that $\cA _{r+1}$ is only defined when $s>0$.


The proof of the next result is analogous to that  of Theorem \ref{thm:classif_nilp}.

\begin{theorem}\label{thm:classif_nilp-complex} Let 
$\cB\in  \g\l(n-1, \C)\subset \g\l(2n-2,\R)$ be a nilpotent matrix and consider $\cA$ as in \eqref{eq:matrixA-complex}, for some $v\in\R^{2n-2}$, with $\mu=0$, where $\cB$ is  the given matrix. Consider the  almost abelian Lie algebra $\g_\cA=\R e_0 \ltimes_\cA \R^{2n-1}$, which admits a complex structure.
\begin{enumerate}
    \item[$\ri$] If $\cB=0$ then $\g_\cA$ is $2$-step nilpotent and isomorphic to  $\g _\cN$, where $\cN$ is as in \eqref{eq:matrixN-complex} with $s=n-1$.
\item[$\rii$] If $\cB\neq 0$ with
\[ {\Sigma}'(\cB)=(r,m_1,\ldots, m_r, p_1,\ldots, p_r, s), \qquad r\geq 1, \;\; m_1 > \dots >m_r\geq 2, \;s\geq 0, \; p_k>0 ,\] 
then there exists a unique integer $\ell$ with  $0\leq \ell \leq r+1-\delta_{s,0}$  such that $\g _\cA$ is isomorphic to $\g _{\cA_\ell}$, where $\cA_0$ and $\cA_\ell$, $1\leq \ell\leq r+1$, are defined in  \eqref{eq:Al-complex}. The Lie algebra  $\g_{\cA_1}$ is $(m_1+1)$-step nilpotent and $\g_{\cA_\ell}$ is $m_1$-step nilpotent for $\ell\neq 1 $.
 \end{enumerate}
 Moreover, the Lie algebras $\g_\cN, \g_{\cA_0}, \dots , \g_{\cA_{r+1}}$ are pairwise non-isomorphic. 
\end{theorem}

\begin{remark} 
We observe that $\g_\mathcal{N}$ is isomorphic to $\h_3\times \R^{2n-3}$, where $\h_3$ is the 3-dimensional Heisenberg Lie algebra.  
\end{remark}

For each $n\geq 2$, let ${\Sigma}' _{n-1}$ denote the set of all possible tuples
\[ (r,m_1,\ldots, m_r, p_1,\ldots, p_r, s), \quad \;\; r>0,\, m_1 > \dots >m_r\geq 2, \;s\geq 0, \; p_k>0 , \] satisfying \eqref{eq:dimn-complex}. We point out that ${\Sigma}' _{n-1}$ parametrizes the conjugacy classes in $\g\l(2n-2,\R)$ of non-zero nilpotent matrices in $\g\l(n-1,\C)\subset \g\l(2n-2,\R)$. According to Theorem \ref{thm:classif_nilp-complex}$\rii$,  each tuple as above gives rise to $r+2-\delta_{s,0}$  matrices in $\g\l(2n-1 , \R) $ of the form \eqref{eq:matrixA-complex}, which correspond  to  different isomorphism classes of $2n$-dimensional nilpotent  almost abelian Lie algebras admitting a complex structure. We denote by $\hat{\Sigma}_{n-1}\subset  \g\l(2n-1 , \R) $ the set of   matrices arising from all possible tuples in $\Sigma'_{n-1}$. 
 It follows  that Lie algebras corresponding to  matrices arising from  different tuples are not isomorphic.  Note that if $\g_\cA$ is a $2n$-dimensional nilpotent almost abelian Lie algebra admitting a complex structure and satisfying   
  $\dim (\ker \cA)=2(n-1)$, then the corresponding matrix $\cB$ is equal to~$0$  and, according to Theorem \ref{thm:classif_nilp-complex}$\ri$,  $\g_\cA$ is isomorphic to $\g_\cN$ with $\cN$ as in \eqref{eq:matrixN-complex} for $s=n-1$. 
  In other words, there is a unique, up to isomorphism, $2n$-dimensional  nilpotent almost abelian  Lie algebra $\g_\cA$ admitting a complex structure and satisfying $\dim (\ker \cA)=2(n-1)$.
   These observations are summarized in the next corollary. 

\begin{corollary}\label{cor:param-nilp-complex}
The isomorphism classes of $2n$-dimensional nilpotent  almost abelian Lie algebras $\g_\cA=\R e_0\ltimes _\cA \R^{2n-1}$  admitting a complex structure and satisfying $\dim (\ker \cA) <2(n-1)$  are parametrized by $\hat{\Sigma}_{n-1}$. If $\dim (\ker \cA) = 2(n-1)$ then $\g_\cA$ is isomorphic to  $\g_\cN$ with $\cN$ as in \eqref{eq:matrixN-complex} for $s=n-1$. 
\end{corollary}

\section{Appendix B}\label{app:isos}

 We include in this section the computation of the Lie group isomorphisms between simply connected almost abelian Lie groups (compare with \cite{AAetal}). Theorem \ref{thm:Lie-iso} below is crucial for the proof of Theorem \ref{thm:p_or_p*}.

Let $G$ be a simply connected almost abelian Lie group with Lie algebra $\g=\R\ltimes_A \R^d$ where $A\in \g\l(d,\R)$. We will show in Proposition \ref{prop:exp} below  that, with few exceptions,  the  exponential map $\exp: \g \to G$ is a diffeomorphism.   The explicit expression of $\exp$ will be given in terms of the  real analytic function $\Phi :\R \to \R$ defined by:
\begin{equation}\label{eq:Phi}
\Phi(x)= \sum_{n=0}^\infty \frac{x^n}{(n+1)!}=  \begin{cases}
\dfrac{e^x-1}{x},&  x\neq 0,\\ \; 1, &  x= 0.
\end{cases}
\end{equation}
We observe that  $\Phi(x)\neq 0$ for all $x\in \R$ and if we extend $\Phi$ to the complex plane, then $\Phi (z)=0$ if and only if $z\in \{ 2k\pi i : k\in \Z, \, k\neq 0\}$. 

Consider on  $\g\l(d,\R)$ the operator norm induced by the euclidean norm on $\R^d$, that is, if  $S\in  \g\l(d,\R)$ then  
$\|S\|=\sup\{ \| S x\| \, :\, x\in \R^d ,\ \| x \| =1\}$. Given a   power series 
$h(x)=\sum_{n=0}^\infty c_n x^n, \, x\in \R$,  such that $\sum_{n=0}^\infty |c_n|\,  r^n$ is convergent for some $r>0$, then for any 
$S\in \g\l(d,\R)$ such that  $\| S \|<r$,  $h(S)=\sum_{n=0}^\infty c_n S^n \in \g\l(d,\R)$ is a well defined operator (see, for instance, \cite[Lemma 3.1.5]{HN}). 

\begin{remark}\label{rem:nu-commuting}
Let $h(x)=\sum_{n=0}^\infty c_n x^n, \, x\in \R$, be a power series   such that $\sum_{n=0}^\infty |c_n|\,  r^n$ is convergent for some $r>0$. Let $S\in \g\l(d,\R)$ and $\nu\neq 0$ such that $\|  S\| <r$ and $\| \nu S\| <r$ and assume that $T\in \g\l(d,\R)$ satisfies $\nu ST= TS$. Then $h\left(\nu S   \right) T = T h(S).$
    
\end{remark}
\begin{lemma}\label{lem:Psi}
    Let $S\in  \g\l(d,\R)$ and $\Phi$ as in \eqref{eq:Phi}. Then $\Phi(S)$ is invertible if and only if $S$ has no eigenvalue in $ \{ 2k\pi i : k\in \Z, \, k\neq 0\}$. 
\end{lemma}
\begin{proof} First we observe that since $\Phi$ is absolutely convergent for all $x\in \R$, $\Phi(S)$ is well defined for any $S\in  \g\l(d,\R)$. The eigenvalues of  $\Phi(S)$ are precisely $\Phi(\lambda)$ with $\lambda$ eigenvalue of $S$. Therefore, $\Phi(S)$ is invertible if and only if $\Phi(\lambda)\neq 0$ for all eigenvalues $\lambda $ of $S$, and since the set of zeros of $\Phi$ is $\{ 2k\pi i : k\in \Z, \, k\neq 0\}$, the lemma follows. 
\end{proof}


\begin{corollary} \label{cor:Psi} Let $S\in  \g\l(d,\R)$ and $\Phi$ as in \eqref{eq:Phi}. Then $\Phi(tS)$ is invertible for all $t\in\R $ if and only if $\lambda \notin\{ a  i : a\in \R, \, a\neq 0\}$ for any eigenvalue $\lambda$  of $S$. 
\end{corollary}


We are now in condition to compute the exponential map on almost abelian Lie groups.

\begin{proposition}\label{prop:exp} 
    Let $G$ be a simply connected almost abelian Lie group with Lie algebra $\g=\R\ltimes_A \R^d$ where $A\in \g\l(d,\R)$. Then $\exp : \g \to G$ is given by: \begin{equation}\label{eq:exp}
        \exp (t,v)=(t, \Phi(tA)v), \qquad (t,v) \in \g, \quad \Phi \text{  as in   (\ref{eq:Phi})}.
    \end{equation}  Moreover,  $\exp$ is a diffeomorphism if and only if   $\lambda \notin\{ a  i : a\in \R, \, a\neq 0\}$ for any eigenvalue $\lambda$  of $A$,  and in this case $\exp ^{-1}$ is given by: 
    \[  \exp ^{-1}(t,w)=(t, \Psi(tA) w),  \quad \text{ where } \Psi (tA)=\Phi(tA)^{-1}.\]  
\end{proposition}
\begin{proof}
    In order to compute $\exp :\g \to G$, we consider the following faithful representations of   $\g$ and $G$ on $\R^{d+1}$:
    \[
\rho (t, v) =  
\left[ \begin{array}{c|c} t A & v \\
\hline
& 0    
\end{array} \right], \quad (t,v)\in \g, \qquad\quad \tau(s,w) =  
\left[ \begin{array}{c|c} e^{s A} & w \\
\hline
& 1    
\end{array} \right], \quad (s,w)\in G. 
    \]
We compute the matrix exponential of $\rho(t,v)$: 
    \[ e^{\rho(t,v)}=\left[ \begin{array}{c|c} e^{t A} & \Phi (tA)v \\
\hline
& 1    
\end{array} \right] =\tau (t, \Phi (tA)v), 
\]
and \eqref{eq:exp} follows since $(d\tau)_{(0,0)}=\rho$. 
The last assertion in the statement is a consequence of Corollary \ref{cor:Psi}.
\end{proof}
\begin{remark} Let $\g$ and $G$ be  as in Proposition \ref{prop:exp}. We point out that when $a  i$ is  an eigenvalue of $A$ for some $a \in \R , \, a\neq 0$, then $\exp:\g \to G$  is neither injective nor surjective.   The non-injectivity is clear.  On the other hand, in order to show that $\exp$ is not surjective, let $r_k=\frac{2\pi k}{a}$ for $k\in \Z, \, k\neq 0$. Then $2\pi k i$ is an eigenvalue of $r_kA$, hence $\Phi (r_kA)$ is not invertible (Lemma \ref{lem:Psi}) and there exists  $w_k \notin \operatorname{Im}\, \Phi (r_kA)$. It follows that $(r_k, w_k)\notin \operatorname{Im}\, \exp$.
\end{remark}

The next corollary will be needed for the proof of Theorem \ref{thm:Lie-iso} below. It is equivalent to the fact that, for each $(t_0,v_0)\in \g$, 
 the map $\gamma :\R \to G$ is a monoparametric subgroup of $G$,  where  $\gamma(t)=\exp \left( t (t_0,v_0) \right)$. 
 \begin{corollary}\label{cor:monop} Let $G$ be a simply connected almost abelian Lie group with Lie algebra $\g=\R\ltimes_A \R^d$ where $A\in \g\l(d,\R)$. Then, for each  $(t_0, v_0)\in \g$, the following equation is satisfied:
 \begin{equation}\label{eq:monop}
     (t+s)\, \Phi\left((t+s)t_0 A\right)v_0=
     t\, \Phi\left(t t_0A\right)v_0+s e^{t\, t_0A}\Phi\left(st_0A\right)v_0, \quad \text{ for all } t, s\in \R,
     \end{equation} 
 where   $\Phi$ is as in   \eqref{eq:Phi}.
 \end{corollary}
 \begin{proof}
   Let  $\gamma(t)=\exp \left( t (t_0,v_0),  \right),\, t\in \R$, then $\gamma$ is a monoparametric subgroup of $G$, that is, $\gamma(t+s)=\gamma(t)\gamma(s)$ for all $t,s\in \R$. We compute each side of this equation, starting with $\gamma(t+s)$:
   \begin{eqnarray*}
     \gamma(t+s)&=& \exp\left( (t+s)t_0, (t+s)v_0   \right)  \\  
     &=& \left( (t+s)t_0, (t+s)\, \Phi\left((t+s)t_0 A\right)v_0   \right).
   \end{eqnarray*}
   On the other hand,
   \begin{eqnarray*}
   \gamma(t)\gamma(s)&=& \exp (tt_0, tv_0) \exp(st_0, sv_0)  \\
   &=&  \left( tt_0, t\, \Phi\left(t t_0A\right)v_0\right) \left(st_0, s\,  \Phi\left(s t_0A\right)v_0\right)\\ 
    &=& \left( (t+s)t_0 , t\, \Phi\left(t t_0A\right)v_0 + s e^{t\, t_0 A}\Phi\left(s t_0A\right)v_0 \right),
   \end{eqnarray*}
   and the corollary follows since $\gamma(t+s)=\gamma(t)\gamma(s)$ for all $t,s\in \R$. 
 \end{proof}


We apply the results above to obtain the Lie group isomorphisms between two simply connected almost abelian Lie groups. We start with the following special case. 

Let $\h_3$ denote the 3-dimensional Heisenberg Lie algebra and $H_3$ the corresponding simply connected Lie group.  If $\g_1$ and $\g_2$ are isomorphic to $\h_3\times \R^{k}$, we describe next the Lie algebra isomorphisms between  $\g_1$ and $\g_2$ and the corresponding Lie group isomorphisms between $G_1$ and $G_2$, the simply connected Lie groups with Lie algebra $\g_1$ and $\g_2$, respectively. In order to state Proposition \ref{prop:h3} below, we fix some notation. Let $\z_1$ be the center of $\g_1$, $e_1,\, \, e_2,\, e_3 \in \g_1$ and fix a subspace $V_1\subset \z_1$ such that
\[
[e_1,e_2]=e_3, \qquad \z_1=\R e_3\oplus V_1 .
\]
 Similarly, let $u_1,\, \, u_2,\, u_3 \in \g_2$ such that $[u_1,u_2]=u_3$ and fix a subspace $V_2\subset \z_2$ such that $\z_2=\R u_3\oplus V_2$, where $\z_2$ is the center of $\g_2$. 

 We point out that $\g_1$ and $\g_2$ are almost abelian Lie algebras, where we fix the following  decompositions: $\g_1 =\R e_1\ltimes (\R e_2 \oplus \z_1)$ and  $\g_2=\R u_1\ltimes (\R u_2 \oplus \z_2)$. The corresponding simply connected Lie groups $G_1$ and $G_2$ decompose both as $\R\ltimes \R^{k+2}$.
\begin{proposition}\label{prop:h3} Let $\g_1$ and $\g_2$ be Lie algebras isomorphic to $\h_3\times \R^{k}$ and let $f:\g_1 \to \g_2$ be a Lie algebra isomorphism. Then, there exist $a_i, b_i, c_i \in \R,\, i=1,2,\; v_1, v_2\in V_2, \, \mu\in V_1^*$ and a linear isomorphism $L: V_1 \to V_2$ such that $\Delta := a_1b_2-a_2b_1\neq 0$ and $f$ is given by:
\begin{eqnarray*}
f(e_i)&=&a_i u_1+b_iu_2+c_iu_3+v_i, \; i=1,2,\\
f(e_3)&= &\Delta u_3, \qquad f(x)=\mu (x)u_3+Lx, \quad x\in V_1.    
\end{eqnarray*}
  The corresponding Lie group isomorphism $F:G_1\to G_2$ such that $(dF)_{(0,0)}=f$ has the following form:
  \[
  F(t,w)=(a_1t+a_2w_2, (a_2t+b_2w_2)u_2+\nu(t,w) u_3+tv_1+w_2v_2+L\widetilde{w}),
  \]
  where $w=w_2e_2+w_3e_3+\widetilde{w}, \, w_i\in \R, \, \widetilde{w}\in V_1$ and $\nu(t,w)\in \R$ is given by:
  \[ 
  \nu(t,w)= \frac 12(a_1t+a_2w_2)(a_2t+b_2w_2)+c_1 t+c_2w_2+\left(w_3-\frac t2 w_2\right)\Delta +\mu(\widetilde{w}).
  \] 
\end{proposition}
\begin{proof}
    Let $f:\g_1 \to \g_2 $ be a Lie algebra isomorphism, then $f$ is an isomorphism between  the centers and the commutator ideals of $\g_1$
    and $\g_2$, respectively. Therefore, $f:\R e_3\oplus V_1 \to \R u_3\oplus V_2$ and $f(e_3)=du_3, $  for some non-zero $d\in \R$. 
    Let  $a_i, b_i, c_i \in \R,\, i=1,2,\; v_1, v_2\in V_2$ such that
    \[
    f(e_i)=a_i u_1+b_iu_2+c_iu_3+v_i, \; i=1,2. 
    \]
    We have that  $du_3=f(e_3)=f([e_1,e_2])=[f(e_1),f(e_2)]=\Delta u_3$ since  $f$ is a Lie algebra isomorphism. Therefore, $d=\Delta$, so in particular, $\Delta\neq 0$. 

    Since $f:\R e_3\oplus V_1 \to \R u_3\oplus V_2$ is a linear isomorphism such that $f(e_3)=\Delta u_3$, there exist  $\mu\in V_1^*$ and a linear isomorphism $L:V_1\to V_2$ such that $f(x)=\mu(x)u_3+Lx$ for all $x\in V_1$. 

    The expression of the corresponding Lie group isomorphism $F$ can be obtained from $f$ since it is well known that $\exp_1 :\g_1\to G_1$ is a diffeomorphism.
\end{proof}


We consider next the general case. Given two  almost abelian  Lie algebras  $\g_1=\R \ltimes _{A_1} \R^d$ and $\g_2=\R \ltimes _{A_2} \R^d$, recall from Lemma \ref{lem:ad-conjugated} that they are isomorphic if there exists $c\neq 0$ and $P\in \operatorname{GL}(d, \R)$ such that $A_1=cPA_2P^{-1}$. 
Let $G_1$ and $G_2$ be the simply connected Lie groups with Lie algebras $\g_1$ and $\g_2$, respectively. The next theorem gives a characterization of the Lie algebra isomorphisms between $\g_1$ and $\g_2$. Moreover, the expression of the corresponding Lie group isomorphisms between $G_1$ and $G_2$ is obtained.
\begin{theorem}\label{thm:Lie-iso}
Let $\g_1=\R \ltimes _{A_1} \R^d$ and $\g_2=\R \ltimes _{A_2} \R^d$ be two isomorphic almost abelian Lie algebras which are not isomorphic to $\h_3 \times \R^{d-2} $ and let $c\in \R, \, c\neq 0 $, $P\in \operatorname{GL}(d, \R)$ such that $A_1=cPA_2P^{-1}$, with $c=1$ when $\g_1$ is nilpotent.  
    \begin{enumerate}
        \item[$\ri$]     
    If $f : \g_1\to \g_2$ is a Lie algebra  isomorphism, then there exist $\mu \in \R, \,\mu\neq 0,$ $v_0\in \R^d$ and 
    $L\in \operatorname{GL}(d, \R)$ such that 
    \begin{equation} \label{eq:Lie-alg-iso}        
     f(t,v)=(\mu t ,Lv+t v_0), \quad  (t,v)\in \g_1 , \quad \text{ with } LPA_2=\frac{\mu}c A_2 LP. 
     \end{equation} 
    Moreover, 
    if $\g_1$ is  not  nilpotent, then $\mu=\pm c$. 
    \item[$\rii$] If $G_1$ and $G_2$ are the simply connected Lie groups with Lie algebras $\g_1$ and $\g_2$, respectively, and $F: G_1\to G_2$ is  the Lie group isomorphism such that $(dF)_{(0,0)}=f$, with $f$ as in \eqref{eq:Lie-alg-iso}, then $F$ has the following form:
\begin{equation} \label{eq:Lie-group-iso} 
F(t,v)=  \left(\mu t, Lv+t \Phi(\mu tA_2) v_0\right), \qquad (t,v)\in G_1 .\end{equation}
\end{enumerate}
\end{theorem}
\begin{proof}
Recall from Remark \ref{rem:nilp-sim} that  if $\g_1$ is nilpotent, then we can take $c=1$. 

Since $\g_1$, hence $\g_2$,   is not isomorphic to $\h^3\times \R ^{d-2}$, then $\g_1$ and $\g_2$ have a unique codimension one ideal (see \cite[Proposition 1]{Fr}). Therefore, if $f: \g_1 \to \g_2$ is a Lie algebra isomorphism, then $f(0\times \R^d)=0\times \R^d$, so  there exist $L\in \operatorname{GL}(d, \R)$, $v_0\in \R^d$  and $\mu\in \R, \, \mu\neq 0,$ such that $f(0,v)=(0,Lv), \, v\in \R^d, $ and $ f(1,0)=(\mu,v_0)$. Since $f$ is a Lie algebra isomorphism we must have $f\left([(1,0),(0,v)]_1\right)=[f(1,0), f(0,v)]_2$, which is equivalent to  
$  LA_1v=\mu A_2Lv$, for all $v\in\R^d$. Since $A_1=cPA_2P^{-1}$, we obtain that 
\begin{equation} \label{eq:A2-conj}
 A_2=\dfrac{c}{\mu} (LP)A_2(LP)^{-1}.\end{equation}

Assume  that $\g_1$ is not nilpotent.   We prove next that $\mu=\pm c$, or equivalently, $\nu=\pm 1$, where $\nu=\dfrac{c}{\mu}$. Let $\{\lambda_1, \dots , \lambda_d\}$, with $|\lambda_1|\leq \cdots \leq |\lambda_d|$, be the (possibly repeated) eigenvalues of $A_2$.  It follows from \eqref{eq:A2-conj} that 
\[ \{\lambda_1, \dots , \lambda_d\}= \{\ \nu\, \lambda_1, \dots , \nu\,\lambda_d\},\]
with $|\nu\, \lambda_1|\leq \cdots \leq |\nu\, \lambda_d|$, therefore, $|\lambda_j|= |\nu\, \lambda_j|$ for all $j$. 
Since $\g_1$ is not nilpotent, $A_1$ and $A_2$ are not nilpotent, in particular, $\lambda_d\neq 0$, which implies that  $|\nu|=1$, that is, $\mu=\pm c$. 
This concludes the proof of $\ri$. 

The proof of $\rii$ consists in showing that $F:G_1\to G_2$ defined as in \eqref{eq:Lie-group-iso} is a Lie group isomorphism satisfying  $(dF)_{(0,0)}=f$. We start by proving that $F$ is a homomorphism. For $(t,v), \, (s,w)\in G_1$ we compute
\begin{eqnarray*}   
F\left((t,v) (s,w)\right)&=& F\left( t+s, v+e^{tA_1} w \right)\\
&=& \left( \mu (t+s), L\left(v+e^{tA_1} w\right) +(t+s) \,\Phi(\mu (t+s) A_2) v_0    \right).
\end{eqnarray*}
On the other hand,
\begin{eqnarray*}   
F(t,v) F(s,w)&=& \left(\mu t, Lv+t \Phi(\mu tA_2) v_0\right) \left(\mu s, Lw+s \Phi(\mu sA_2) v_0\right) \\
&=& \left(\mu(t+s), Lv+t \Phi(\mu tA_2) v_0+ e^{\mu tA_2}\left( Lw+s \Phi(\mu sA_2) v_0 \right)  \right). 
\end{eqnarray*}
Since \eqref{eq:monop} holds for $A=A_2$ and $(t_0, v_0)=(\mu, v_0)$, it follows that 
 $F$ is a group homomorphism if and only if 
\begin{equation}\label{eq:exp-LA1}
Le^{tA_1}w= e^{\mu t A_2}Lw, \text{ for all } w\in \R^d, \quad \text{ that is, }\; Le^{tA_1}= e^{\mu t A_2}L.
\end{equation}
This follows from $A_1=cPA_2P^{-1}$, since 
\[
Le^{tA_1}=LP e^{ctA_2}P^{-1}=e^{\mu tA_2}L,
\]
where the last equality holds by recalling from  \eqref{eq:A2-conj} that $LPA_2= \dfrac {\mu}c A_2LP$ and applying Remark \ref{rem:nu-commuting} with $h(x)=e^x, \, T=LP, \, S= c tA_2, \, \nu= \dfrac{\mu}c $. Therefore, $F$ is a Lie group homomorphism. We show next that $(dF)_{(0,0)}=f$, that is, $F\circ \exp_1=\exp_2 \circ f$, where $\exp_j:\g_j\to G_j$ denotes the exponential map, $j=1,2$.

First we observe that the same argument used to obtain \eqref{eq:exp-LA1}, but with $h(x)=\Phi(x)$, gives:
\begin{equation} \label{eq:LA1} L \Phi(tA_1)=\Phi\left(\mu t A_2\right) L. 
\end{equation}
Now, for $(t,v)\in \g_1$ we compute: 
\begin{eqnarray*}
F( \exp_1(t,v))&=& F(t, \Phi\left(tA_1)v\right)= \left( \mu t, L\Phi(tA_1)v +t \Phi(\mu tA_2) v_0\right)\\ 
&=& \left( \mu t, \Phi\left(\mu t A_2\right) Lv +t \Phi(\mu tA_2) v_0\right)\\ 
&=& \left( \mu t, \Phi\left(\mu t A_2\right)\left( Lv +tv_0\right)\right)=\exp_2 (\mu t, Lv +tv_0)\\
&=& \exp_2 (f(t,v)),
\end{eqnarray*}
where the third equality follows from \eqref{eq:LA1}. 
Therefore, $F\circ \exp_1=\exp_2 \circ f$, 
 as asserted. 
In particular, $(dF)_{(0,0)}$ is an isomorphism, which implies that $(dF)_{(s,w)}$ is an isomorphism for all $(s,w)\in G_1$, since $F$ is a homomorphism. It is clear that $F$ is bijective, hence, $F$ is a Lie group isomorphism. 
\end{proof}




\medskip

\begin{thebibliography}{99}\frenchspacing

\bibitem{AAetal}
M. Almora, Z. Avetisyan, K. Berlow, I. Martin, G. Rakholia, K. Yang, H. Zhang, Z. Zhao, Almost Abelian Lie groups, subgroups and quotients, \textit{J. Math. Sci., New York, Series A} \textbf{266} (2022), 42--65.

\bibitem{AO}
A. Andrada, M. Origlia, Lattices in almost abelian Lie groups with locally conformal K\"ahler or symplectic structures, \textit{Manuscripta Math.} \textbf{155} (2018), 389--417.

\bibitem{AB}
A.~Andrada, M.~L.~Barberis, Hypercomplex almost abelian solvmanifolds, \textit{J. Geom. Anal.}  \textbf{33} (2023), Article 213. 

\bibitem{AT}
A.~Andrada, A.~Tolcachier, Harmonic almost complex structures on almost abelian Lie groups and solvmanifolds, \textit{Ann. Mat. Pura Appl.} \textbf{203} (2024), 1037--1060.

\bibitem{AdBM}
A.~Andrada, V.~del Barco, A.~Moroianu, Locally conformally product structures on solvmanifolds, \textit{Ann. Mat. Pura Appl.} \textbf{203} (2024),  2425--2456.


\bibitem{AV}
Z. Avetisyan, R. Verch, Explicit harmonic and spectral analysis in Bianchi I-VII-type cosmologies, \textit{
Class. Quantum Grav.} \textbf{30} (2013), Article  155006.

\bibitem{Av}
Z. Avetisyan, 
The structure of almost abelian Lie algebras, \textit{
Int. J. Math.} \textbf{33} (2022), Article  2250057.

\bibitem{Bar}
M. L. Barberis, Hypercomplex structures on four-dimensional Lie groups, \textit{Proc. Amer. Math. Soc.} \textbf{125} (1997), 1043--1054.

\bibitem{BeFi}
L.-B. Beaufort, A. Fino, 
Locally conformal SKT almost abelian Lie algebras, \textit{
Linear Algebra Appl.} \textbf{684} (2024), 1--22.  

\bibitem{BGV}
L. Bedulli, G. Gentili, L.  Vezzoni, 
A parabolic approach to the Calabi-Yau problem in HKT geometry, 
\textit{Math. Z.} \textbf{302} (2022), 917--933.

\bibitem{Bo} 
C.~Bock, On low-dimensional solvmanifolds, \textit{Asian J. Math.} \textbf{20} (2016), 199--262.

\bibitem{Boy}
C.~Boyer, A note on hyperhermitian four-manifolds, \textit{
Proc.  Amer. Math. Soc.}
\textbf{102} (1988),  157--164.

\bibitem{Cal}
E. Calabi,  M\'etriques K\"ahl\'eriennes et fibr\'es holomorphes, \textit{Ann. Sci. Ec. Norm. Super.} \textbf{12} (1979), 269--294. 

\bibitem{C-M}
L.~P.~Castellanos Moscoso, 
Left-invariant symplectic structures on diagonal almost abelian Lie groups, \textit{
Hiroshima Math. J.} \textbf{52} (2022), 357--378.

\bibitem{CM}
S. Console, M. Macrì, 
Lattices, cohomology and models of 6-dimensional almost abelian solvmanifolds, 
\textit{Rend. Semin. Mat., Univ. Politec. Torino} \textbf{74} (2016),  95--119.

\bibitem{DF}
I. Dotti, A. Fino,  Hyper-K\"ahler with torsion
structures invariant by nilpotent Lie groups, \textit{Class. Quantum Grav.} \textbf{19} (2002), 1--12.

\bibitem{DF1}
I.~Dotti, A.~Fino, Hypercomplex eight-dimensional nilpotent Lie groups, \textit{J.Pure Appl. Algebra} \textbf{184} (2003), 41--57.

\bibitem{FG}
A. Fino, G. Grantcharov, Properties of manifolds with skew-symmetric
torsion and special holonomy, \textit{Adv.  Math.} \textbf{189} (2004), 439--450.

\bibitem{FiPa1}
A.~Fino, F.~Paradiso, 
Generalized Kähler almost abelian Lie groups, \textit{
Ann. Mat. Pura Appl.} \textbf{ 200} (2021), 1781--1812.

\bibitem{FiPa2}
A.~Fino, F.~Paradiso, 
Balanced Hermitian structures on almost abelian Lie algebras, \textit{
J. Pure Appl. Algebra} \textbf{ 227} (2023), Article  107186.

\bibitem{Fr}
M.~Freibert, Cocalibrated structures on Lie algebras with a codimension one Abelian ideal, \textit{Ann. Glob. Anal. Geom.} \textbf{42} (2012), 537--563. 

\bibitem{GeTa}
G. Gentili, N. Tardini, HKT manifolds: Hodge theory, formality and balanced metrics,  \textit{Q. J. Math.} \textbf{75} (2024),  
413--435.

\bibitem{GLV}
G. Grantcharov, M. Lejmi, M. Verbitsky, Existence of HKT metrics on hypercomplex manifolds of real dimension 8, \textit{Adv. Math.} \textbf{320} (2017), 1135--1157.

\bibitem{HO}
C.~Herrera, M.~Origlia, 
Invariant conformal Killing forms on almost abelian Lie groups, preprint 2024, arXiv:2402.09229.

\bibitem{HN}
J. Hilgert, K.-H. Neeb, \textit{Structure and Geometry of Lie Groups}, Springer Monographs in Mathematics, Springer, 2012.

\bibitem{HKLR}
N.~J.~Hitchin, A.~Karlhede, U.~Lindström, M.~Roček, 
Hyperkähler metrics and supersymmetry, \textit{Comm. Math. Phys.}~\textbf{108} (1987), 535--589.

\bibitem{HP}
P. S. Howe, G. Papadopoulos, Twistor spaces for hyper-K\"ahler manifolds with torsion, \textit{Phys. Lett. B} \textbf{379} (1996),  80--86.

\bibitem{Hu}
H.~Huang, Lattices and harmonic analysis on some 2-step solvable Lie groups, \textit{J. Lie 
Theory} \textbf{13} (2003), 77--89.

\bibitem{IP}
S. Ivanov, A. Petkov, HKT manifolds with holonomy $\operatorname{SL}(n,\H)$, \textit{Int. Math. Res. Not.} \textbf{2012} (2012), 3779--3799.

\bibitem {joy}  
D.~Joyce,  Compact hypercomplex and quaternionic manifolds,
\textit{ J. Differential Geom.} \textbf{35} (1992), 743--761.

\bibitem{Ka}
I.~Kaplansky, \textit{Linear Algebra and Geometry. A second course},   Allyn and Bacon, Inc.,  1969.

\bibitem{Ku}
D.~Kurtz, A sufficient condition for all the roots of a polynomial to be real, \textit{Amer. Math. Monthly} \textbf{99} (1992), 259--263.

\bibitem{LRV} 
J.~Lauret, E.~Rodr\'iguez-Valencia, On the Chern-Ricci flow and its solitons for Lie groups, \textit{Math. Nachr.} \textbf{288} (2015), 1512--1526.

\bibitem{MeTa}
M.~Lejmi, N.~Tardini, 
On the invariant and anti-invariant cohomologies of hypercomplex manifolds, \textit{Transform. Groups} (2023), https://doi.org/10.1007/s00031-023-09828-x 

\bibitem{Mal}
A.~Malcev, On a class of homogeneous spaces, \textit{Izv. Akad. Nauk SSSR} \textbf{ 13} (1949),  9--32; English translation in \textit{Am. Math. Soc. Transl.} \textbf{ 39} (1951), 33p.

\bibitem{Mi} 
J.~Milnor, Curvatures of left invariant metrics on Lie groups, \textit{Adv. Math.} \textbf{21} (1976), 293--329.

\bibitem{Mo}
A.~J.~Moreno, 
Harmonic $G_2$-structures on almost abelian Lie groups, \textit{Differ. Geom. Appl.} \textbf{ 91} (2023), Article  102060.

\bibitem{Mos}
G.~D.~Mostow, Factor spaces of solvable groups, \textit{Ann. Math.} \textbf{60} (1954), 1--27.

\bibitem{new}  
A.~Newlander, L.~Nirenberg,  Complex analytic
coordinates in almost complex manifolds, \textit{Ann. Math.} \textbf{65}
(1957), 391--404.

\bibitem{Ob}
M.~Obata, Affine connections on manifolds with almost complex, quaternion or Hermitian structure, \textit{Japanese J.  Math.} \textbf{26} (1956), 43--79.



\bibitem{Rod}
L.~Rodman, \textit{Topics in quaternion linear algebra},  
Princeton Series in Applied Mathematics, Princeton University Press, 2014.

\bibitem{RS}
M.~P.~Ryan, L.~C.~Shepley, \textit{Homogeneous Relativistic Cosmologies}, Princeton Series in Physics \textbf{59}, Princeton University Press, 2015.

\bibitem{Sai}
M.~Saito, Sur certains groupes de Lie résolubles II, \textit{Sci. Papers Coll. Gen. Ed. Univ. Tokyo} \textbf{7} (1957), 157--168.

\bibitem{Sol}
A.~Soldatenkov, Holonomy of the Obata connection on $\operatorname{SU}(3)$, \textit{
Int. Math. Res. Not.} \textbf{2012} (2012), 3483--3497.

\bibitem{SV}
A.~Soldatenkov, M.~Verbitsky, 
Subvarieties of hypercomplex manifolds with holonomy in $\operatorname{SL}(n,\H)$,
\textit{J. Geom. Phys. } \textbf{62} (2012),  2234--2240.

\bibitem{SSTVP} 
Ph.~Spindel, A.~Sevrin, W.~Troost, A.~Van Proeyen, Extended super-symmetric $\sigma$-models on group manifolds, \textit{Nuclear Phys. B} \textbf{  308} (1988), 662--698.


\bibitem{V}
M.~Verbitsky, 
Balanced HKT metrics and strong HKT metrics on hypercomplex manifolds, \textit{Math. Res. Lett.} \textbf{16} (2009),  735--752.
\end{thebibliography}
\end{document}